\documentclass[a4paper,11pt]{article}

\usepackage[utf8x]{inputenc}
\usepackage[T1]{fontenc}
\usepackage{lmodern}

\usepackage[colorlinks]{hyperref}
\hypersetup{linkcolor=blue, citecolor=red}

\usepackage{bbm}
\usepackage[bbgreekl]{mathbbol}
\usepackage{nccrules}
\usepackage[nottoc]{tocbibind}
\usepackage[intlimits,leqno]{amsmath}
\usepackage{amsthm}
\usepackage{amssymb}
\usepackage{mathrsfs}
\usepackage{stmaryrd}
\usepackage{xspace}
\usepackage[all]{xy}

\usepackage{paralist} \setdefaultitem{$\star$}{}{}{}

\usepackage{geometry}
\newcommand{\Z}{\ensuremath{\mathbb{Z}}}
\newcommand{\Q}{\ensuremath{\mathbb{Q}}}
\newcommand{\R}{\ensuremath{\mathbb{R}}}
\newcommand{\C}{\ensuremath{\mathbb{C}}}

\newcommand{\Tr}{\ensuremath{\mathrm{tr}\xspace}}

\newcommand{\Sym}{\ensuremath{\mathrm{Sym}}}

\newcommand{\dd}{\ensuremath{\,\mathrm{d}}}

\newcommand{\angles}[1]{\ensuremath{\langle #1 \rangle}}
\newcommand{\mes}{\ensuremath{\mathrm{mes}}}

\newcommand{\Stab}{\ensuremath{\mathrm{Stab}}}

\newcommand{\identity}{\ensuremath{\mathrm{id}}}

\newcommand{\Hom}{\ensuremath{\mathrm{Hom}}}

\newcommand{\rightiso}{\ensuremath{\stackrel{\sim}{\rightarrow}}}

\newcommand{\Ker}{\ensuremath{\mathrm{ker}\xspace}}

\newcommand{\Ad}{\ensuremath{\mathrm{Ad}\xspace}}

\newcommand{\Spec}{\ensuremath{\mathrm{Spec}\xspace}}

\newcommand{\mult}{\ensuremath{\mathrm{mult}}}

\newcommand{\Gm}{\ensuremath{\mathbb{G}_\mathrm{m}}}
\newcommand{\Ga}{\ensuremath{\mathbb{G}_\mathrm{a}}}

\newcommand{\Supp}{\ensuremath{\mathrm{Supp}}}

\newcommand{\utimes}[1]{\ensuremath{\overset{#1}{\times}}}

\newcommand{\GL}{\ensuremath{\mathrm{GL}}}

\newcommand{\GSpin}{\ensuremath{\mathrm{GSpin}}}

\newcommand{\PGL}{\ensuremath{\mathrm{PGL}}}

\newcommand{\SL}{\ensuremath{\mathrm{SL}}}
\newcommand{\Sp}{\ensuremath{\mathrm{Sp}}}
\newcommand{\GSp}{\ensuremath{\mathrm{GSp}}}

\newcommand{\MSp}{\ensuremath{\mathrm{MSp}}}

\theoremstyle{plain}
\newtheorem{proposition}{Proposition}[subsection]
\newtheorem{lemma}[proposition]{Lemma}
\newtheorem{theorem}[proposition]{Theorem}
\newtheorem{corollary}[proposition]{Corollary}

\theoremstyle{definition}
\newtheorem{definition}[proposition]{Definition}
\newtheorem{definition-theorem}[proposition]{Definition-Theorem}
\newtheorem{definition-proposition}[proposition]{Definition-Proposition}
\newtheorem{hypothesis}[proposition]{Hypothesis}
\newtheorem{example}[proposition]{Example}

\theoremstyle{remark}
\newtheorem{remark}[proposition]{Remark}

\renewcommand{\Re}{\ensuremath{\mathrm{Re}\xspace}}

\title{Basic functions and unramified local $L$-factors for split groups}
\author{Wen-Wei Li}
\date{}

\begin{document}

\maketitle

\begin{abstract}
  According to a program of Braverman, Kazhdan and Ngô Bao Châu, for a large class of split unramified reductive groups $G$ and representations $\rho$ of the dual group $\hat{G}$, the unramified local $L$-factor $L(s,\pi,\rho)$ can be expressed as the trace of $\pi(f_{\rho,s})$ for a suitable function $f_{\rho,s}$ with non-compact support whenever $\Re(s) \gg 0$. Such functions can be plugged into the trace formula to study certain sums of automorphic $L$-functions. It also fits into the conjectural framework of Schwartz spaces for reductive monoids due to Sakellaridis, who coined the term \textit{basic functions}; this is supposed to lead to a generalized Tamagawa-Godement-Jacquet theory for $(G,\rho)$. 
  In this article, we derive some basic properties for the basic functions $f_{\rho,s}$ and interpret them via invariant theory. In particular, their coefficients are interpreted as certain generalized Kostka-Foulkes polynomials defined by Panyushev. These coefficients can be encoded into a rational generating function.
\end{abstract}

\tableofcontents

\section{Introduction}\label{sec:intro}

\paragraph{History}
Let $F$ be a non-archimedean local field, $\mathfrak{o}_F$ be its ring of integers, and choose a uniformizer $\varpi$ of the maximal ideal of $\mathfrak{o}_F$. Denote by $q_F$ the cardinality of the residue field of $F$. Let $G$ be an unramified connected reductive $F$-group. Fix a hyperspecial maximal compact subgroup $K$ of $G(F)$. In this article we shall always assume $G$ split. Let $\hat{G}$ be the complex dual group of $G$, and $\rho: \hat{G} \to \GL(V,\C)$ be a finite-dimensional algebraic representation. The unramified local $L$-factor attached to these objects is defined by
$$ L(s,\pi,\rho) := \det(1 - \rho(c)q_F^{-s}|V)^{-1} \quad \in \C(q_F^s), \; s \in \C $$
for an unramified irreducible representation $\pi$ of $G(F)$ of Satake parameter $c \in \hat{T}/W$ where $\hat{T}$ is a maximal torus of $\hat{G}$ and $W$ is the corresponding Weyl group. It is the basic building block of automorphic $L$-functions via Euler products.

The first and perhaps the best studied example is the \textit{standard $L$-factor} of Tamagawa \cite{Ta63}, Godement and Jacquet \cite{GJ72}. It corresponds to the case where $G = \GL(n)$ and $\rho = \text{Std}: \GL(n,\C) \to \GL(n,\C)$ is the \textit{standard representation}, i.e.\! the identity map. Their approach is to consider the function $\mathbbm{1}_{\text{Mat}_{n \times n}(\mathfrak{o}_F)}$, identified with its restriction to $\GL(n,F)$. The integral pairing between $\mathbbm{1}_{\text{Mat}_{n \times n}(\mathfrak{o}_F)}$ and the zonal spherical function of $\pi \otimes |\det|_F^s$ yields $L\left(s - \frac{n-1}{2}, \pi, \text{Std}\right)$ whenever $\Re(s) \gg 0$. Their result can be paraphrased as follows: the Satake transform $\mathcal{S}(\mathbbm{1}_{\text{Mat}_{n \times n}(\mathfrak{o}_F)})$ equals the rational function $c \mapsto \det\left( 1 - \rho(c)q_F^{-s+(n-1)/2} \big| V \right)^{-1}$ on $\hat{T}/W$. 
Some generalization of the Satake isomorphism is needed, as $\mathbbm{1}_{\text{Mat}_{n \times n}(\mathfrak{o}_F)}$ is not compactly supported on $\GL(n,F)$ and $\det(1 - \rho(\cdot)q_F^{-s+(n-1)/2} \big| V)^{-1}$ is not a regular function. This is not a serious issue, however (see \S\ref{sec:ac}).

How about other pairs $(G,\rho)$? Satake \cite[Appendix 1]{Sat63} and Shimura \cite{Shi63} independently tried some other classical groups such as $\GSp(4)$ embedded in the monoid $\MSp(4)$, that is, its Zariski closure in $\text{Mat}_{4 \times 4}$. Here the determinant character of $\GL(n)$ is replaced by the similitude character $\GSp(4) \to \Gm$. It turns out that the function $\mathbbm{1}_{\MSp(4, \mathfrak{o}_F)}$ does not produce an $L$-factor. 

Braverman and Kazhdan \cite{BK00} explored the idea of a generalized Godement-Jacquet theory from the other side. Roughly speaking, they considered a short exact sequence of connected reductive groups
\begin{gather}\label{eqn:ses-0}
  1 \to G_0 \to G \xrightarrow{\det_G} \Gm \to 1
\end{gather}
and a representation $\rho: \hat{G} \to \GL(V,\C)$ such that $\rho$ restricted to $\Gm \subset \hat{G}$ is $z \mapsto z\cdot\identity$. In the unramified case, they started from the observation
\begin{gather}\label{eqn:L-sym-0}
  L(s, \pi, \rho) = \sum_{k \geq 0} \Tr(\Sym^k \rho(c)) q_F^{-ks}, \quad \Re(s) \gg 0.
\end{gather}
In this framework, a distinguished $K$-bi-invariant function $f_\rho$ is defined by taking the sum $\sum_{k \geq 0} \mathcal{S}^{-1}(\Tr(\Sym^k \rho))$. This is a well-defined function on $G(F)$ since the $k$-th summand is supported on $\{g \in G(F) : |\det_G(g)|_F = q_F^{-k} \}$. Note that it is never compactly supported on $G(F)$. We deduce that
\begin{gather}\label{eqn:f_rho-L-0}
  \Tr\left(\pi \otimes |\det_G|_F^s \right) (f_\rho) = L(s, \pi, \rho)
\end{gather}
for every unramified irreducible representation $\pi$, whenever $\Re(s) \gg 0$.

In the Tamagawa-Godement-Jacquet setting, one recovers $f_{\text{Std}} = \mathbbm{1}_{\text{Mat}_{n\times n}(\mathfrak{o}_F)} |\det|_F^{(n-1)/2}$. In general, however, almost nothing has been said about $f_\rho$ beyond its existence.

These constructions can be understood in terms of a program of Sakellaridis \cite[\S 3]{Sake12}, who emphasized the importance of the conjectural \textit{Schwartz space} $\mathcal{S}(X)$ attached to a spherical $G$-variety $X$ in harmonic analysis. In the unramified local setting, there should exist a distinguished element called the \textit{basic function} in $\mathcal{S}(X)$ whose behaviour reflects the singularities of $X$. In the preceding cases, the relevant spherical varieties are expected to be some \textit{reductive monoids} \cite{Vi95,Re05} containing $G$ as the unit group, such as $\text{Mat}_{n \times n} \supset \GL(n)$ or $\MSp(2n) \supset \GSp(2n)$.

In \cite{NgSa,Ng12} Ngô Bao Châu formulates a precise construction of the objects in \eqref{eqn:ses-0} and an irreducible representation $\rho: \hat{G} \to \GL(V,\C)$. The inputs are
\begin{inparaenum}[(i)]
  \item a simply connected split semisimple $F$-group $G_0$ and
  \item a dominant element $\bar{\xi}$ in $X_*(T_\text{ad})$, where $T_\text{ad}$ denotes a maximal torus in the adjoint group $G_\text{AD}$.
\end{inparaenum}
The relevant monoid $M_\xi$ here is constructed from $\bar{\xi}$ using Vinberg's \textit{enveloping monoids} \cite[Theorem 5]{Vi95}. The basic function $f_\rho$ is defined by inverting the Satake transform $\mathcal{S}$ as before so that \eqref{eqn:f_rho-L-0} is satisfied. In the equi-characteristic setting, Ngô conjectures that $f_\rho$ comes from some perverse sheaf on $M_\xi$ via the function-sheaf dictionary. Note that reductive monoids are usually singular \cite[Theorem 27.25]{Ti11}.

The significance of the functions $f_\rho$ may be partially explained by the fact that they can be plugged into the Arthur-Selberg trace formula upon some twist $f_{\rho,s} = f_\rho |\det_G|_F^s$ with $\Re(s) \gg 0$ (see \cite{FLM11} for the delicate analytic issues). This will permit us to express the partial automorphic $L$-function as a trace, and it has applications to Langlands' program of \textit{beyond endoscopy}. Cf. \cite{Ng12} and Matz's thesis \cite{Matz}. 

Another direction is pioneered by L.\! Lafforgue \cite{La13} in his search of a kernel for Langlands functoriality, in which the basic function $f_\rho$ is an instance of his functions of $L$-type. The crucial ingredient thereof, the conjectural non-linear Poisson summation formula for certain reductive monoids, is also a prominent part in the work of Braverman and Kazhdan \cite{BK00}. Using the Plancherel formula, however, Lafforgue is able to formulate the relevant Fourier transform in a precise manner. Note that a wider class of representations $\rho$ than Ngô's construction is needed, and he considered quasi-split groups as well.

In all the aforementioned works, there is no description of the basic function without resort to the Satake transform, except in the case $(G,\rho)=(\GL(n),\text{Std})$. This makes it difficult to understand the behaviour of $f_\rho$. In the next paragraph we will see where the obstacle lies: it is related to the decomposition of symmetric power representations and the Kazhdan-Lusztig polynomials.

We recommend the nice survey by Casselman \cite{Cas14} for this circle of ideas.

\paragraph{Our results}
Now we may state our main results. Let $G$, $\det_G$, $\rho$, etc. be as in \eqref{eqn:ses-0}. Choose a suitable Borel pair $(B,T)$ of $G$ and consider the Cartan decomposition $G(F) = K T(F)_+ K$ using the anti-dominant Weyl chamber $X_*(T)_-$ in the cocharacter lattice $X_*(T)$, where $T(F)_+$ is the image of $X_*(T)_-$ under $\mu \mapsto \mu(\varpi)$. Let $\rho_{B^-}$ be the half-sum of negative roots relative to $(B,T)$. The homomorphism $\det_G: G \twoheadrightarrow \Gm$ induces a homomorphism $\det_G: X_*(T) \twoheadrightarrow \Z$.

The basic function is determined by its restriction to $T(F)_+$. We shall introduce an indeterminate $X$ in place of $q_F^{-s}$ and write
\begin{align*}
  f_{\rho, X} & = \sum_{\mu \in X_*(T)_-} c_\mu(q_F) q_F^{-\angles{\rho_{B^-}, \mu}} \mathbbm{1}_{K\mu(\varpi)K} \cdot  X^{\det_G(\mu)} \\
  & = \sum_{\mu \in X_*(T)_-} c_\mu(q_F) \delta_{B^-}^{\frac{1}{2}}(\mu(\varpi)) \mathbbm{1}_{K\mu(\varpi)K} \cdot  X^{\det_G(\mu)}
\end{align*}
instead of $f_{\rho,s} = f_\rho |\det_G|_F^s$. Here $c_\mu(q_F)$ are certain polynomials in $q_F^{-1}$. Let $\leq$ be the Bruhat order relative to the opposite Borel subgroup $B^-$. For every $\lambda \in X_*(T)_-$, let $V(\lambda)$ denote the irreducible representation of $\hat{G}$ of highest weight $\lambda$. Then $c_\mu(q) \in \Z[q^{-1}]$ is given by
\begin{gather*}
  c_\mu(q) := \begin{cases}\displaystyle
    \sum_{\substack{\lambda \in X_*(T)_- \\ \lambda \geq \mu}} q^{-\angles{\rho_{B^-}, \lambda-\mu}} P_{n_\mu, n_\lambda}(q) \mult(\Sym^{\det_G(\mu)} \rho : V(\lambda)), & \text{if } \det_G(\mu) \geq 0, \\
    0, & \text{otherwise}.
  \end{cases}
\end{gather*}
Here $P_{n_\mu, n_\lambda}(q)$ are some Kazhdan-Lusztig polynomials: they appear in the Kato-Lusztig formula (Theorem \ref{prop:Kato-Lusztig}) for $\mathcal{S}^{-1}$.

In particular, $c_\mu(q)$ has non-negative integral coefficients. Also note that $P_{n_\mu, n_\lambda}(q)$ reduces to weight-multiplicities when $q=1$. These properties already imply some easy estimates for $c_\mu(q_F)$ upon passing to the ``classical limit'' $q=1$ (see \S\ref{sec:estimates}). Nonetheless, for the study of the structural properties of the basic function $f_{\rho, X}$, we need more.

The formula for $c_\mu(q)$ above suggests that one has to understand
\begin{enumerate}
  \item the decomposition of $\Sym^k\rho$ into irreducibles, for all $k$;
  \item the Kazhdan-Lusztig polynomial $P_{n_\mu,n_\lambda}(q)$ for infinitely many $\lambda$, since $\Sym^k \rho$ produces infinitely many irreducible constituents as $k$ varies.
\end{enumerate}
Both tasks are daunting, but their combination turns out to have a nice interpretation. We shall make use of \textit{invariant theory} for the dual group $\hat{G}$, for want of anything better. Suppose $\lambda, \mu \in X_*(T)_-$. Set $K_{\lambda, \mu}(q) = q^{\angles{\rho_{B^-}, \lambda-\mu}} P_{n_\mu, n_\lambda}(q^{-1})$. One can show that $K_{\lambda, \mu}(q) \in \Z_{\geq 0}[q]$, known as Lusztig's $q$-analogue of weight-multiplicities or the Kostka-Foulkes polynomial, for the based root datum of $\hat{G}$. An argument due to Hesselink (see \cite{Bry89}) expresses $K_{\lambda, \mu}(q)$ as the Poincaré series in $q$ of the vector space
$$ \Hom_{\hat{G}} \left( V(\lambda), \Gamma\left( \hat{G} \utimes{\hat{B}} \hat{\mathfrak{u}}, \; p^* \mathscr{L}_{\hat{G}/\hat{B}}(\mu) \right) \right), $$
which is endowed with the $\Z_{\geq 0}$-grading coming from the dilation action of $\Gm$ along the fibers of $p: \hat{G} \utimes{\hat{B}} \hat{\mathfrak{u}} \to \hat{G}/\hat{B}$. Note that the homogeneous fibration $\hat{G} \utimes{\hat{B}} \hat{\mathfrak{u}} \to \hat{G}/\hat{B}$ may also be identified with the cotangent bundle $T^*(G/B)$. Plugging this into the formula for $c_\mu(q^{-1})$, we arrive at the Poincaré series of
$$ \Hom_{\hat{G}} \left( \Sym V, \Gamma\left( \hat{G} \utimes{\hat{B}} \hat{\mathfrak{u}}, \; p^* \mathscr{L}_{\hat{G}/\hat{B}}(\mu) \right) \right), $$
or equivalently, $(\C[V] \otimes \Gamma(\cdots))^{\hat{G}}$. Constructions of this type were familiar to the invariant theorists in the  nineteenth century, known as the space of \textit{covariants}.

In order to study $f_{\rho, X}$, we form the formal power series
$$ \mathbf{P} := \sum_{\mu \in X_*(T)_-} c_\mu(q^{-1}) e^\mu X^{\det_G(\mu)} $$
and interpret it as a multi-graded Poincaré series of the space of invariants $\mathcal{Z}^{\hat{G}}$ of some infinite-dimensional $\C$-algebra $\mathcal{Z}$ with $\hat{G} \times \hat{T} \times \Gm^2$-action. Here $\hat{T}$ (resp. $\Gm^2$) is responsible for the grading corresponding to the variable $\mu$ (resp. by $(X,q)$) in the Poincaré series. This is our Theorem \ref{prop:basicfun-gen}.

As a byproduct, we get another description of the coefficients $c_\mu(q^{-1}) \in \Z[q]$ of $f_{\rho, X}$: for $\mu \in X_*(T)_-$, they equal $q^{-\det_G(\mu)}$ times the \textit{generalized Kostka-Foulkes polynomials} $m^\mu_{0, \Psi}(q)$ defined à la Panyushev \cite{Pan10}, attached to the data
\begin{compactitem}
  \item the Borel subgroup $\hat{B}$ of $\hat{G}$;
  \item $N = V \oplus \hat{\mathfrak{u}}$, as a representation of $\hat{B}$ where $\hat{\mathfrak{u}}$ is acted upon by the adjoint representation;
  \item $\Psi$ is the set with multiplicities of the $\hat{T}$-weights of $N$.
\end{compactitem}

These polynomials are defined in a purely combinatorial way: for $\lambda \in X_*(T)_-$ and $\mu \in X_*(T)$, set
$$ m^\mu_{\lambda, \Psi}(q) := \sum_{w \in W} (-1)^{\ell(w)} \mathcal{P}_\Psi \left( w(\lambda + \check{\rho}_{B^-}) - (\mu + \check{\rho}_{B^-}); q \right) $$
where $\mathcal{P}_\Psi$ is a $q$-analogue of Kostant's partition function, defined by
\begin{gather}
  \prod_{\alpha \in \Psi} (1 - qe^{-\alpha}) = \sum_{\nu \in X_*(T)} \mathcal{P}_\Psi(\nu; q) e^\nu.
\end{gather}

For $N = \hat{\mathfrak{u}}$ we recover $K_{\lambda, \mu}(q)$. This gives a surprising combinatorial formula for $c_\mu(q^{-1})$, but it is ill-suited for computation. See \S\ref{sec:com-setup} for details.

One can regard $\mathbf{P}$ as a generating function for $c_\mu(q^{-1})$. Generating functions are most useful when they are \textit{rational}, and this becomes clear in the invariant-theoretic setup: $\mathcal{Z}$ is identified with the coordinate ring of an affine variety $\mathcal{Y}$, thereby showing the rationality of $\mathbf{P}$. This is our Theorem \ref{prop:c_mu-rat-fun}, the main qualitative result of this article. In fact $\mathbf{P}$ takes the form
$$
   \frac{Q}{\prod_{i=1}^s (1 - q^{d_i} e^{\mu_i} X^{\det_G(\mu_i)})}
$$
for some $s$, where $d_i \geq 0$, $\mu_i \in X_*(T)_-$, $\det_G(\mu_i) \geq 0$, and $Q$ is a $\Z$-linear combination of monomials of the form $e^\mu X^{\det_G(\mu)} q^d$.

Broer \cite{Bro93} employed a similar strategy to study $K_{\lambda, \mu}(q)$ for semisimple groups.

For the explicit determination of the rational Poincaré series $\mathbf{P}$ of $\mathcal{Z}^{\hat{G}}$, i.e. the determination of $f_{\rho, X}$, we have to know the homogeneous generators (or: the $d_i$, $\mu_i$ in the denominator) and the syzygies (or: the numerator $Q$) thereof. In general, this seems to be a difficult problem of invariant theory. In this article we accomplish this only for $(\GL(n), \text{Std})$ in \S\ref{sec:GJ}. Conceivably, tools from computational invariant theory might have some use here.

Our results might also shed some light on the definition of Schwartz spaces for the monoids arising from Ngô's recipe. Nevertheless, the monoid-theoretic aspects are deliberately avoided in this article. Hopefully they will be treated in subsequent works.

\subsection*{Postscript}
In an earlier version, the Brylinski-Kostant filtration \cite{Bry89} of $V$ was used to interpret $\mathbf{P}$ as a Poincaré series. Due to some careless manipulations of filtered vector spaces under tensor products and confusions about the definition of Brylinski-Kostant filtrations, the resulting formula is wrong except in the $(\GL(n), \text{Std})$ case. I would like to thank Professor Casselman for pointing this out.

\subsection*{Acknowledgements}
I am grateful to Bill Casselman, Laurent Lafforgue and Chung Pang Mok for helpful conversations. In fact, my interest in this topic is initiated by the lectures of Casselman and C.\! P.\! Mok in Beijing, 2013. The \S\ref{sec:GSp4} on $\GSp(4)$ is heavily influenced by Casselman's lectures. My thanks also go to Nanhua Xi for his unquestionable expertise in Kazhdan-Lusztig theory.

\subsection*{Organization of this article}
In \S\ref{sec:Satake-review} we collect the basic properties of the Satake isomorphism for split groups, including the description of its inverse in terms of the Kazhdan-Lusztig polynomials. We also introduce an easy yet handy generalization, namely the almost compactly supported version of the Satake isomorphism.

In the subsequent sections we switch to the framework of anti-dominant weights. It makes the invariant-theoretic results in \S\ref{sec:gKF} much cleaner.

We revert to the harmonic analysis for $p$-adic groups in \S\ref{sec:L-basic}. We review Ngô's recipe, the definition for the basic functions $f_{\rho,X}$ (as well as their specializations $f_{\rho}, f_{\rho,s}$), and their relation to unramified local $L$-factors. The coefficients of the basic functions are explicitly expressed in terms of
\begin{inparaenum}[(i)]
  \item the Kazhdan-Lusztig polynomials, and
  \item the multiplicities in the decomposition of symmetric power representations.
\end{inparaenum}
Then we give several estimates for $f_{\rho, s}$ by passing to the ``classical limit'', namely by specialization to $q=1$. For example, $f_{\rho,s}$ is shown to be tempered in the sense of Harish-Chandra when $\Re(s) \geq 0$

In order to study qualitative behaviour of basic functions, we set up the combinatorial and geometric formalism for the generalized Kostka-Foulkes polynomials in \S\ref{sec:gKF}, following Panyushev \cite{Pan10}. The upshot is the interpretation of these polynomials as certain multi-graded Poincaré series. Unlike \cite{Bry89,Bro93,Pan10}, we allow connected reductive groups and use anti-dominant weights in our exposition. Complete proofs will be given in order to dispel any doubt.

In \S\ref{sec:gKF-basic}, we interpret the coefficients $c_\mu(q^{-1})$ of the basic function as certain generalized Kostka-Foulkes polynomials $m^\mu_{0,\Psi}(q)$. We are then able to encode those $c_\mu(q^{-1})$ into a rational Poincaré series; this is largely based on the arguments in \cite{Bro93}.

In \S\ref{sec:examples}, as a reality check, we first consider the Tamagawa-Godement-Jacquet construction for the standard $L$-factor of $\GL(n)$. We recover their function $\mathbbm{1}_{\text{Mat}_{n \times n}}$ as $f_{\text{Std},-\frac{n-1}{2}}$. In the case of the spinor $L$-factor for $\GSp(4)$, we compute the based root datum and the weights of the standard representation of $\GSp(4,\C)$ explicitly. It turns out that from the viewpoint of $L$-factors, our basic function is ``more basic'' then the one considered by Satake and Shimura. Unfortunately we are not yet able to determine $f_\rho$ completely in the $(\GSp(4), \text{spin})$ case.

\subsection*{Conventions}
\paragraph{Local fields}
Throughout this article, $F$ always denotes a non-archimedean local field. Denote its ring of integers by $\mathfrak{o}_F$, and choose a uniformizer $\varpi$ of the maximal ideal of $\mathfrak{o}_K$. Set $q_F := |\mathfrak{o}_F/(\varpi)|$, the cardinality of the residue field of $F$.

Denote the normalized valuation of $F$ by $\text{val}: F \to \Z \sqcup \{+\infty\}$. The normalized absolute value of $F$ is $|\cdot|_F := q_F^{-\text{val}(\cdot)}: F \to \Q$.

\paragraph{Groups and representations}
Let $\Bbbk$ be a commutative ring with $1$. For a $\Bbbk$-group scheme $G$, the group of its $\Bbbk$-points is denoted by $G(\Bbbk)$. The algebra of regular functions on $G$ is denoted by $\Bbbk[G]$. Assume henceforth that $\Bbbk$ is a field, the center of $G$ will then be denoted by $Z_G$. When $\Bbbk$ is algebraically closed, the algebraic groups over $\Bbbk$ are identified with their $\Bbbk$-points.

The derived group of $G$ is denoted by $G_\text{der}$. Now assume $G$ to be connected reductive. The simply connected cover of $G_\text{der}$ is denoted by $G_\text{SC} \twoheadrightarrow G_\text{der}$. We denote the adjoint group of $G$ by $G_\text{AD}$, equipped with the homomorphism $G \twoheadrightarrow G_\text{AD}$. For every subgroup $H$ of $G$, we denote by $H_\text{sc}$ (resp. $H_\text{ad}$) the preimage of $H$ in $G_\text{SC}$ (resp. image in $G_\text{AD}$). For example, if $T$ is a maximal torus of $G$, then $T_\text{ad}$ is a maximal torus of $G_\text{AD}$.

A Borel pair of $G$ is a pair of the form $(B,T)$ where $B$ is a Borel subgroup and $T \subset B$ is a maximal torus; we shall always assume that $B$ and $T$ are defined over the base field $\Bbbk$. Once a Borel pair $(B,T)$ is chosen, the opposite Borel subgroup $B^-$ is well-defined: it satisfies $B \cap B^- = T$. The Weyl group is denoted by $W := N_G(T)/T$. The longest element in $W$ is denoted by $w_0$. The length function of $W$, or more generally of any extended Coxeter systems, is denoted by $\ell(\cdot)$.

When $\Bbbk = F$ is a non-archimedean local field, the representations of $G(F)$ are always assumed to be smooth and admissible. When $\Bbbk$ is algebraically closed, the representations of $G$ are always assumed to be algebraic and finite-dimensional; in this case we denote the representation ring by $\text{Rep}(G)$. For an algebraic finite-dimensional algebraic representation $(\rho, V)$ of $G$, we may define the set \textit{with multiplicities} of its $T$-weights, denoted as $\Supp(V)$. 

The Lie algebra of $G$ (resp. $B$, etc.) is denoted by $\mathfrak{g}$ (resp. $\mathfrak{b}$, etc.) as usual.

If $H$ is a locally compact group, the modulus function $\delta_H: H \to \R_{>0}$ is the character defined by $\mu(h \cdot h^{-1}) = \delta_H(h)\mu(\cdot)$, where $\mu$ is any Haar measure on $H$.

\paragraph{Varieties with group action}
Let $\Bbbk$ be an algebraically closed field and $G$ be an algebraic group over $\Bbbk$. Varieties over $\Bbbk$ are irreducible by convention. By a $G$-variety we mean an algebraic $\Bbbk$-variety $X$ equipped with the action (on the left) morphism $a: G \times X \to X$ such that $a|_{1 \times X} = \identity$ and satisfying the usual associativity constraints. The action will often be abbreviated as $gx = a(g,x)$. We may also talk about morphisms between $G$-varieties, etc.

On a $G$-variety $X$ we have the notion of $G$-linearized quasi-coherent sheaves, that is, a quasi-coherent sheaf $\mathcal{F}$ equipped with an isomorphism $\text{pr}_2^* \mathcal{F} \rightiso a^* \mathcal{F}$ inducing a $G$-action on the set of local sections, where $\text{pr}_2: G \times X \to X$ is the second projection. A $G$-linearization of $\mathcal{F}$ induces $G$-actions on the cohomology groups $H^i(X, \mathcal{F})$, for every $i$. See \cite[Appendix C]{Ti11} for details.

We will occasionally deal with hypercohomology of complexes of sheaves on $X$; the relevant notations will be self-evident.

\paragraph{Combinatorics}
Always fix a base field $\Bbbk$. For a $\Bbbk$-torus $T$, we write $X^*(T) := \Hom(T,\Gm)$, $X_*(T) := \Hom(\Gm, T)$ where the $\Hom(\cdots)$ is taken in the category of $\Bbbk$-tori and $\Gm$ denotes the multiplicative $\Bbbk$-group scheme. We identify $\Z$ with $\Hom(\Gm, \Gm)$ by associating $k$ to the homomorphism $z \mapsto z^k$. Then the composition of homomorphisms gives a duality pairing $\angles{\cdot, \cdot}: X^*(T) \otimes X_*(T) \to \Z$.

Let $G$ be a split connected reductive $\Bbbk$-group with a Borel pair $(B,T)$. The Weyl group $W$ acts on $X^*(T)$ and $X_*(T)$ so that $\angles{\cdot,\cdot}$ is $W$-invariant. Define
\begin{compactitem}
  \item $\Delta_B$: the set of simple roots relative to $B$,
  \item $\Delta_B^\vee$: the set of simple coroots relative to $B$,
  \item $\Sigma_B$: the set of positive roots relative to $B$,
  \item $\Sigma_B^\vee$: the set of positive coroots relative to $B$,
  \item $\rho_B$: the half sum of the elements of $\Sigma_B$,
  \item $\check{\rho}_B$: the half sum of the elements of $\Sigma_B^\vee$.
\end{compactitem}

Note that $\Sigma_B \subset X^*(T)$. For each $\alpha \in \Sigma_B$, the corresponding coroot is denoted by $\alpha^\vee \in X_*(T)$. The Bruhat order on $X^*(T)$ relative to $B$ is defined by $\lambda_1 \leq \lambda_2$ if and only if $\lambda_2 = \lambda_1 + \sum_{\alpha \in \Delta_B} n_\alpha \alpha$ with $n_\alpha \geq 0$ for all $\alpha$. Similarly, the Bruhat order on $X_*(T)$ relative to $B$ is defined by the requirement that $\mu_1 \leq \mu_2$ if and only if $\mu_2 = \mu_1 + \sum_{\alpha \in \Delta_B} n_\alpha \alpha^\vee$ for non-negative $n_\alpha$.

The dominant cone in $X_*(T)$ relative to $B$ is defined as
$$ X_*(T)_+ := \{ \mu \in X_*(T) : \angles{\alpha,\mu} \geq 0 \text{ for all } \alpha \in \Delta_B \}. $$
Likewise, using $B^-$ one defines the anti-dominant cone
$$ X_*(T)_- := \{ \mu \in X_*(T) : \angles{\alpha,\mu} \geq 0 \text{ for all } \alpha \in \Delta_{B^-} \}. $$

Set $X_*(T)_\R := X_*(T) \otimes_\Z \R$, etc. The cones $X_*(T)_{\R,\pm}$ are defined as before. Similarly, the dominant and anti-dominant cones in $X^*(T)$ are defined using coroots $\alpha^\vee$ instead of $\alpha$.

\paragraph{Dual groups}
Always assume $G$ split. The dual group of $G$ in the sense of Langlands will be denoted by $\hat{G}$; it is defined over an algebraically closed field $\Bbbk$ of characteristic zero, usually $\Bbbk=\C$. In this article, the group $G$ will always be equipped with a Borel pair $(B,T)$. Therefore $\hat{G}$ is equipped with the dual Borel pair $(\hat{B},\hat{T})$. Write the based root datum of $G$ as $(X^*(T), \Delta_B, X_*(T), \Delta_B^\vee)$, the based root datum of $\hat{G}$ is given by
$$ (X^*(\hat{T}), \Delta_{\hat{B}}, X_*(\hat{T}), \Delta_{\hat{B}}^\vee) = (X_*(T), \Delta_B^\vee, X^*(T), \Delta_B). $$

For the dual group $\hat{G}$, or more generally for a connected reductive group over an algebraically closed field of characteristic zero, the irreducible representations are classified by their highest weights relative to $B$. If $\xi \in X^*(\hat{T})_+ = X_*(T)_+$, the corresponding irreducible representation of $\hat{G}$ is denoted by $V(\xi)$. We will consider the case relative to $B^-$ as well, in which the highest weights belong to $X_*(T)_-$.

The character of a representation $\rho \in \text{Rep}(\hat{G})$ is denoted by $\Tr(\rho)$. It can be identified with its restriction to $\hat{T}$, hence with the $W$-invariant element of the group ring $\Z[X^*(\hat{T})] = \Z[X_*(T)]$ given by
$$ \Tr(\rho) = \sum_{\nu \in X^*(\hat{T})} \dim V_\nu \cdot e^\nu $$
where $V_\mu$ is the $\mu$-weight subspace of $V$. We may also view $\Tr(\rho)$ as a regular function on the variety $\hat{T}/W$.

\paragraph{Miscellany}
For any vector space $V$ over a base field, its linear dual is denoted by $V^\vee$. The symmetric (resp. exterior) $k$-th power is denoted by $\Sym^k V$ (resp. $\bigwedge^k V$). The same notation pertains to representations and sheaves.

The trace of a trace class operator is denoted by $\Tr(\cdots)$.

For any sets $Y \subset X$, we denote by $\mathbbm{1}_Y: X \to \{0,1\}$ the characteristic function of $Y$.

The symmetric group on $k$ letters is denoted by $\mathfrak{S}_k$.

Let $Y$ be a commutative monoid. Its monoid ring is denoted as $\Z[Y]$; for example, taking $Y=\Z_{\geq 0} \cdot q$ furnishes the polynomial ring $\Z[q]$. Write the binary operation of $Y$ additively, the corresponding elements in $\Z[Y]$ are formally expressed in exponential notations: $\{ e^y : y \in Y \}$, subject to $e^{y_1 + y_2} = e^{y_1}e^{y_2}$. In particular $1 = e^0$ is the unit element in $\Z[Y]$. 

We will use the standard notations $\GL(n)$, $\GSp(2n)$, etc. to denote the general linear groups, symplectic similitude groups, etc. Let $A$ be a ring, the $A$-algebra of $n \times n$-matrices will be denoted by $\text{Mat}_{n \times n}(A)$.

\section{Review of the Satake isomorphism}\label{sec:Satake-review}
The materials here are standard. We recommend the excellent overview \cite{Gr98}.

\subsection{The Satake transform and $L$-factors}\label{sec:Satake-trans}
Consider the following data
\begin{itemize}
  \item $G$: a split connected reductive $F$-group,
  \item $(B,T)$: a Borel pair of $G$ defined over $F$,
  \item $U$: the unipotent radical of $B$,
  \item $\hat{G}$: the dual group over $\C$ of $G$, equipped with the dual Borel pair $(G,T)$.
\end{itemize}

Fix a hyperspecial vertex in the Bruhat-Tits building of $G$ that lies in the apartment determined by $T$. It determines a hyperspecial maximal compact subgroup $K$ of $G(F)$. Note that
$$ K_T := T(F) \cap K $$
is a hyperspecial subgroup of $T(F)$.

Choose the Haar measure on $G(F)$ satisfying $\mes(K)=1$. An admissible smooth representation of $G(F)$ is called $K$-unramified if it contains nonzero $K$-fixed vectors.

Let $T(F)_- \subset T(F)$ be the image of $X_*(T)_+$ under the map $\mu \mapsto \mu(\varpi)$. We have the Cartan decomposition
$$ G(F) = K T(F)_- K. $$

The integral $K$-spherical Hecke algebra $\mathcal{H}(G(F)\sslash K; \Z)$ is defined as the convolution algebra of \textit{compactly supported} functions $K \backslash G(F) / K \to \Z$. It has the $\Z$-basis $\mathbbm{1}_{K \mu(\varpi) K}$ parametrized by $\mu \in X_*(T)_+$. More generally, for every ring $R$, we define $\mathcal{H}(G(F)\sslash K; R)$ by considering bi-invariant functions under $K$ with values in $R$. Equivalently,
$$ \mathcal{H}(G(F)\sslash K; R) = \mathcal{H}(G(F)\sslash K; \Z) \otimes_\Z R . $$

The same definitions also apply to $T(F)$ with respect to $K_T$. The Weyl group $W$ acts on $\mathcal{H}(T(F)\sslash K_T; R)$.

\begin{definition}[I.\! Satake \cite{Sat63}]
  Let $R$ be a $\Z[q_F^{\pm \frac{1}{2}}]$-algebra. The \textit{Satake isomorphism} is defined as the homomorphism between $R$-algebras
  \begin{align*}
    \mathcal{S}: \mathcal{H}(G(F)\sslash K; R) & \stackrel{\sim}{\longrightarrow} \mathcal{H}(T(F)\sslash  K_T; R)^W, \\
    f & \longmapsto \left[ t \mapsto \delta_B(t)^{\frac{1}{2}} \int_{U(F)} f(tu) \dd u \right],
  \end{align*}
  where $U(F)$ is equipped with the Haar measure such that $\mes(U(F) \cap K) = 1$. Also recall that $\delta_B^{1/2}: \mu(\varpi) \mapsto q_F^{-\angles{\rho_B, \mu}}$ takes value in $R$ and factors through $K_T$.
\end{definition}

Note that $X_*(T) \rightiso T(F)/K_T$ by $\mu \mapsto \mu(\varpi)$, by which we have $\mathcal{H}(T(F)\sslash  K_T; \Z)^W = \Z[X_*(T)]^W $. On the other hand, it is well-known that $\text{Rep}(\hat{G}) = \Z[X^*(\hat{T})]^W = \Z[X_*(T)]^W$: to each representation $\rho \in \text{Rep}(\hat{G})$ we attach its character $\Tr(\rho) \in \Z[X_*(T)]^W$.

We will be mainly interested in the usual $K$-spherical Hecke algebra
$$ \mathcal{H}(G(F)\sslash K) := \mathcal{H}(G(F)\sslash K; \C). $$
Elements of $\mathcal{H}(T(F)\sslash  K_T)^W = \text{Rep}(\hat{G}) \otimes_\Z \C$ can then be identified as regular functions on the $\C$-variety $\hat{T}/W$. In this setting, $\mathcal{S}$ establishes bijections between
\begin{inparaenum}[(i)]
  \item the classes $c \in \hat{T}/W$,
  \item the $1$-dimensional representations of the algebra $\mathcal{H}(T(F)\sslash K_T)^W = \C[X_*(T)]^W$;
  \item the isomorphism classes of $K$-unramified irreducible representations of $G(F)$.
\end{inparaenum}
Let $\pi_c$ be the $K$-unramified representation corresponding to a class $c$, then the bijection is characterized by
$$ \Tr\;\pi_c(f) = \mathcal{S}(f)(c), \quad \varphi \in \mathcal{H}(G(F)\sslash K). $$

On the other hand, $c$ can be identified with a unramified character $\chi_c$ of $T(F)$ (unique up to $W$-action). Denote the normalized parabolic induction of $\chi_c$ from $B$ as $I_B(\chi_c)$. Then $\pi_c$ can also be characterized as the $K$-unramified constituent of $I_B(\chi_c)$. Note that $I_B(\chi_c)$ is irreducible for $c$ in general position. It follows that the Satake isomorphism $\mathcal{S}$ does not depend on the choice of $B$.

We say that $c \in \hat{T}/W$ is the Satake parameter of the $K$-unramified irreducible representation $\pi$ if $\pi = \pi_c$.

Now comes the unramified $L$-factor. Fix a representation
$$ \rho: \hat{G} \to \GL(V,\C) $$
of the dual group $\hat{G}$ over $\C$. Let $c \in \hat{T}/W$. It can also be viewed as an element in the adjoint quotient of  $\hat{G}$ by Chevalley's theorem.

\begin{definition}
  Introduce an indeterminate $X$. The unramified local $L$-factor attached to $\pi_c$ and $\rho$ is defined by
  $$ L(\pi_c, \rho, X) := \det(1 - \rho(c)X|V)^{-1} \quad \in \C(X). $$

  The usual $L$-factors are obtained by specializing $X$, namely
  $$ L(s, \pi_c, \rho) := L(\pi_c, \rho, q_F^{-s}), \quad s \in \C,$$
  which defines a rational function in $q_F^{-s}$. In what follows we will omit the underlying spaces $V$, etc. in the traces.

  The following alternative description is well-known: see \cite[(2.6)]{Mc95}
  \begin{gather}\label{eqn:L-Sym}
    L(\pi_c, \rho, X) = \left[ \sum_{i=0}^{\dim V} (-1)^i \Tr\left( \bigwedge^i \rho(c)\right) X^i \right]^{-1} = \sum_{k \geq 0} \Tr\left( \Sym^k \rho(c) \right) X^k \quad \in \C\llbracket X\rrbracket.
  \end{gather}
\end{definition}

\subsection{Inversion via Kazhdan-Lusztig polynomials}\label{sec:Satake-inversion}
We set out to state the Kato-Lusztig formula for the inverse of $\mathcal{S}$. The basic references are \cite{Lu83,Ka82}; a stylish approach can be found in \cite{HKP10}.

Let us introduce an indeterminate $q$. For any $\nu \in X_*(T)$, we define $\mathcal{P}(\nu; q) \in \Z[q]$ by
$$ \prod_{\alpha \in \Sigma_B} (1 - qe^{\check{\alpha}})^{-1} = \sum_{\nu \in X_*(T)} \mathcal{P}(\nu; q) e^\nu. $$

For $\lambda, \mu \in X_*(T)$, define
\begin{gather}\label{eqn:K}
  K_{\lambda, \mu}(q) := \sum_{w \in W} (-1)^{\ell(w)} \mathcal{P}(w(\lambda + \check{\rho}_B) - (\mu + \check{\rho}_B); q) \quad \in \Z[q].
\end{gather}
Kostant's partition function is recovered by taking the ``classical limit'' $q = 1$. For this reason it is called Lusztig's $q$-analogue.

At this point, it is advisable to clarify the relation between $K_{\lambda, \mu}(q)$ and certain Kazhdan-Lusztig polynomials, since the latter appear frequently in the relevant literature. Consider the extended affine Weyl group $\tilde{W} := X_*(T) \rtimes W$, regarded as a group of affine transformations on $X_*(T) \otimes_\Z \R$. As is well-known, the standard theory of affine Coxeter groups carries over to this setting. We will be sketchy at this point. First, recall the definition of the set of simple affine roots $S_\text{aff} = S \sqcup \{s_0\}$. Here $S \subset W$ is the set of simple root reflections determined by $B$ and $s_0$ is the reflection whose fixed locus is defined by $\angles{\tilde{\alpha}, \cdot}=1$, with $\tilde{\alpha}$ being the highest root.

One can then write $\tilde{W} = W_\text{aff} \rtimes \Omega$ where $(W_\text{aff}, S_\text{aff})$ is an authentic affine Coxeter system and $\Omega$ is the normalizer of $S_\text{aff}$. The length function $\ell$ and the Bruhat order $\leq$ can then be extended to $\tilde{W}$ by stipulating
\begin{itemize}
  \item $\ell(wz)=\ell(w)$ for $w \in W_\text{aff}$, $z \in \Omega$;
  \item $wz \leq w'z'$ if and only if $w \leq w' \in W$, $z=z' \in \Omega$.
\end{itemize}

The \textit{Kazhdan-Lusztig polynomials} $P_{w,w'}(q) \in \Z[q]$ (see \cite{KL79}) can also be defined on $\tilde{W}$: we have $P_{wz,w'z'} \neq 0$ only when $z=z' \in \Omega$, and $P_{wz,w'z} = P_{w,w'}$.

For every $\mu \in X_*(T)$, there exists a longest element $n_\mu$ in $W \mu W$; in fact it equals $w_0 \mu$.

\begin{theorem}[{\cite[Theorem 1.8]{Ka82}}]\label{prop:K-P}
  Let $\mu, \lambda \in X_*(T)_+$ such that $\lambda \geq \mu$. Then
  $$ K_{\lambda, \mu}(q) = q^{\angles{\rho_B, \lambda-\mu}} P_{n_\mu, n_\lambda}(q^{-1}). $$
\end{theorem}

Now we can invert the Satake isomorphism. Given $\lambda \in X_*(T)_+ = X^*(\hat{T})_+$, recall that $V(\lambda)$ denotes the irreducible representation of $\hat{G}$ of highest weight $\lambda$ relative to $B$, and the character $\Tr\,V(\lambda)$ is regarded as an element of $\Z[X_*(T)]^W = \mathcal{H}(T(F)\sslash K_T; \Z)^W$.

\begin{theorem}[{\cite[(3.5)]{Ka82}} or {\cite[Theorem 7.8.1]{HKP10}}]\label{prop:Kato-Lusztig}
  Let $R$ be a $\Z[q_F^{\pm \frac{1}{2}}]$-algebra and $\lambda \in X_*(T)_+ = X^*(\hat{T})_+$, then
  \begin{align*}
    \Tr V(\lambda) &= \sum_{\substack{\mu \in X_*(T)_+ \\ \mu \leq \lambda}} q_F^{-\angles{\rho_B, \mu}} K_{\lambda,\mu}(q_F^{-1}) \mathcal{S}(\mathbbm{1}_{K\mu(\varpi)K}) \\
    & = \sum_{\substack{\mu \in X_*(T)_+ \\ \mu \leq \lambda}} q_F^{-\angles{\rho_B, \lambda}} P_{n_\mu, n_\lambda}(q_F) \mathcal{S}(\mathbbm{1}_{K\mu(\varpi)K}).
  \end{align*}
  as elements of $\mathcal{H}(T(F)\sslash K_T; R)^W$.
\end{theorem}

Since $\text{Rep}(\hat{G}) = \mathcal{H}(T(F)\sslash K_T; \Z)^W$, the Theorem does give the inverse of $\mathcal{S}$.

\begin{remark}
  Theorem \ref{prop:Kato-Lusztig} is sometimes stated under the assumption that $\hat{G}$ is adjoint, so that $\tilde{W} = W_\text{aff}$. The case for split reductive groups is covered in \cite[\S 10]{HKP10}.
\end{remark}

\subsection{Functions of almost compact support}\label{sec:ac}
We record a mild generalization of the Satake isomorphism here. Retain the same assumptions on $G$, $B$, $T$ and $K$. Set
$$ \mathfrak{a}^*_G := \Hom_{\text{alg.grp}}(G, \Gm) \otimes_\Z \R $$
and let $\mathfrak{a}_G$ be its $\R$-linear dual. We define Harish-Chandra's homomorphism $H_G: G(F) \to \mathfrak{a}_G$ (with the same sign convention as in \cite[p.240]{Wa03}) as the homomorphism characterized by
$$ q_F^{-\angles{\chi, H_G(\cdot)}} = |\chi(\cdot)|_F, \quad \chi \in X^*(G). $$
The image of $H_G$ is a lattice in $\mathfrak{a}_G$ denoted by $\mathfrak{a}_{G,F}$. Observe that $H_G$ is zero on $U(F)$ and $K$.

Fix a $\Z[q_F^{\pm \frac{1}{2}}]$-algebra $R$. We denote by $C_c(\mathfrak{a}_{G,F}; R)$ the $R$-module of finitely supported function $\mathfrak{a}_{G,F} \to R$. The following notion of functions of \textit{almost compact support} (abbreviation: ac) is borrowed from Arthur; it will also make sense in the archimedean case.

\begin{definition}
  Given functions $f: G(F) \to R$ and $b \in C_c(\mathfrak{a}_{G,F}; R)$, we write
  $$ f^b(\cdot) := b(H_G(\cdot)) f(\cdot): G(F) \to R. $$
  Define $\varphi^b$ similarly for $\varphi: T(F) \to R$ and $b \in C_c(\mathfrak{a}_{G,F}; R)$. Set
  \begin{align*}
    \mathcal{H}_\text{ac}(G(F)\sslash K; R) & := \left\{ f: G(F) \to R, \forall b \in C_c(\mathfrak{a}_{G,F}; R), \; f^b \in \mathcal{H}(G(F)\sslash K; R) \right\}, \\
    \mathcal{H}_\text{ac}(T(F)\sslash K_T; R) & := \left\{ f: T(F) \to R, \forall b \in C_c(\mathfrak{a}_{G,F}; R), \; f^b \in \mathcal{H}(T(F)\sslash K_T; R) \right\}.
  \end{align*}
\end{definition}

The left and right convolution products endow $\mathcal{H}_\text{ac}(G(F)\sslash K; R)$ (resp. $\mathcal{H}_\text{ac}(T(F)\sslash K_T; R)$) with a $\mathcal{H}(G(F)\sslash K; R)$-bimodule (resp. $\mathcal{H}(T(F)\sslash K_T; R)$-bimodule) structure. The Weyl group $W$ acts on $\mathcal{H}_\text{ac}(T(F)\sslash K_T; R)$ as usual.

\begin{proposition}
  The Satake isomorphism $\mathcal{S}$ extends to an isomorphism between $R$-modules
  $$ \mathcal{S}: \mathcal{H}_\mathrm{ac}(G(F)\sslash K; R) \rightiso \mathcal{H}_\mathrm{ac}(T(F)\sslash K_T; R)^W $$
  characterized by
  $$ \mathcal{S}(f^b) = \mathcal{S}(f)^b $$
  for any $b \in C_c(\mathfrak{a}_{G,F}; R)$ and $f \in \mathcal{H}_\mathrm{ac}(G(F)\sslash K; R)$. Moreover, it respects the $\mathcal{H}(G(F)\sslash K; R)$-bimodule (resp. $\mathcal{H}(T(F)\sslash K_T; R)^W$-bimodule) structures.
\end{proposition}
\begin{proof}
  Since $H_G$ is $W$-invariant and $H_G \equiv 0$ on $U(F)$, the integral defining $\mathcal{S}(f)$ still makes sense for $f \in \mathcal{H}_\text{ac}(G(F)\sslash K; R)$ and we have indeed $\mathcal{S}(f^b) = \mathcal{S}(f)^b$. It characterizes $\mathcal{S}(f)$ by an argument of partition of unity on $\mathfrak{a}_{G,F}$. The preservation of bimodule structures is routine to check.
\end{proof}

The inverse of $\mathcal{S}$ in the almost compactly supported setting is given by exactly the same formulas as in Theorem \ref{prop:Kato-Lusztig}. When $R=\C$, we write $\mathcal{H}_\text{ac}(G(F) \sslash K)$, etc.

Note that elements in $\mathcal{H}_\text{ac}(T(F)\sslash K_T)^W$ are not necessarily regular functions on $\hat{T}/W$. It contains some formal functions, as we will see later on.

\section{Unramified $L$-factors and the basic function}\label{sec:L-basic}

\textbf{Caution} -- Henceforth we shall use the opposite Borel subgroup to define various objects. More precisely, for a given split connected reductive $F$-group $G$ with the Borel pair $(B,T)$,
\begin{compactitem}
  \item the Bruhat order in $X_*(T)$ is taken relative to $B^-$ unless otherwise specified;
  \item the highest weight of an irreducible representation of $\hat{G}$, etc. is now taken relative to $B^-$;
  \item consequently, in the polynomials $P_{n_\mu,n_\lambda}(q)$ and $K_{\lambda,\mu}(q)$ we assume $\lambda \in X_*(T)_-$;
  \item we use the Cartan decomposition relative to $B_-$, so that $\{ \mathbbm{1}_{K\mu(\varpi)K} : \mu \in X_*(T)_- \}$ will form a basis of $\mathcal{H}(G(F)\sslash K)$;
  \item the image of $X_*(T)_-$ under $\mu \mapsto \mu(\varpi)$ is denoted by $T(F)_+$;
  \item in parallel, the irreducible characters $\{ \Tr V(\lambda) : \lambda \in X_*(T)_- \}$ of $\hat{G}$ form a basis of $\mathcal{H}(T(F) \sslash K_T)^W$.
\end{compactitem}
Nevertheless, the Satake isomorphism $\mathcal{S}$ is independent of the choice of Borel subgroup.

\subsection{Ngô's recipe}\label{sec:Ngo}
Here we give a brief review of \cite{NgSa}.

Let $G_0$ be a split unramified $F$-group which is semi-simple and simply-connected. Fix a Borel pair $(B_0, T_0)$ for $G_0$ and define the dual avatars $\widehat{G_0}$, $\widehat{B_0}$, $\widehat{T_0}$ over $\C$. Let $Z_0 := Z_{G_0}$.

Given $\bar{\xi} \in X_*(T_{0,\text{ad}})_- = X^*(\widehat{T_0}_{,\text{sc}})_-$, we deduce an irreducible representation
$$ \rho_{\bar{\xi}}: \widehat{G_0}_{,\text{SC}} \to \GL(V,\C) $$
of highest weight $\bar{\xi}$ relative to $\widehat{B_0}^-$.

Note that $\widehat{G_0}$ is an adjoint group. The highest weight $\bar{\xi}$ for $\rho_{\bar{\xi}}$ is not always liftable to $X_*(T_0)_-$, thus what we have for $\widehat{G_0}$ is just a projective representation $\bar{\rho}: \widehat{G_0} \to \PGL(V,\C)$. It can be lifted to an authentic representation upon passing to a canonical central extension $\hat{G}$ by $\Gm$, as explicated by the following commutative diagram
\begin{gather}\label{eqn:lifting}
\xymatrix{
  1 \ar[r] & \Gm \ar[r]^{\widehat{\det_G}} \ar@{=}[d] & \hat{G} \ar[r] \ar[d]^{\rho} \ar@{}[rd]|{\Box} & \widehat{G_0} \ar[d]^{\bar{\rho}} \ar[r] & 1 \\
  1 \ar[r] & \Gm \ar[r] & \GL(V,\C) \ar[r] & \PGL(V,\C) \ar[r] & 1
}\end{gather}
in which the rows are exact and the rightmost square is cartesian. Since $\bar{\rho}$ can also be lifted to $\widehat{G_0}_{,\text{SC}}$, there is another description for $\hat{G}$: denote the central character of $\rho_{\bar{\xi}}$ by $\omega_{\bar{\xi}}$, then we have
$$ \hat{G} = \dfrac{\widehat{G_0}_{,\text{SC}} \times \Gm }{\left\{ (z^{-1}, \omega_{\bar{\xi}}(z)) : z \in Z_{\widehat{G_0}_{,\text{SC}}} \right\} } . $$

The complex group $\hat{G}$ inherits the Borel pair $(\hat{B}, \hat{T})$ from $(\widehat{B_0}, \widehat{T_0})$. Dualization  gives a short exact sequence of split unramified $F$-groups
\begin{gather}
  1 \to G_0 \to G \xrightarrow{\det_G} \Gm \to 1
\end{gather}
and a Borel pair $(B,T)$ for $G$. It induces a short exact sequence
$$ 0 \to X_*(T_0) \to X_*(T) \xrightarrow{\det_G} X_*(\Gm) \to 1. $$

Hereafter, we shall forget $G_0$ and work exclusively with $G$, $B$, $T$, the homomorphism $\det_G: G \to \Gm$ and the representation $\rho: \hat{G} \to \GL(V,\C)$. Note that
\begin{enumerate}
  \item $\rho$ is irreducible with a highest weight $\xi \in X^*(\hat{T})_- = X_*(T)_-$ relative to $B^-$, it is mapped to $\bar{\xi} \in X_*(T_{0,\text{ad}})_-$ via $G \twoheadrightarrow G_{0,\text{AD}}$;
  \item the restriction of $\rho$ on $\Gm \hookrightarrow \hat{G}$ is simply $z \mapsto z\cdot\identity$, this means that $\rho$ satisfies \cite[(3.7)]{BK00} with respect to $\det_G: G \to \Gm$;
  \item since $\widehat{G_0}$ is adjoint, $\Ker(\rho)$ is always connected, hence $\rho$ is admissible in the sense of \cite[Definition 3.13]{BK00}: this follows from the construction of $\hat{G}$ as a fiberd product;
  \item every weight $\mu \in X_*(T) = X^*(\hat{T})$ of $\rho$ satisfies $\det_G(\mu)=1$.
\end{enumerate}

\begin{remark}
  For a similar construction for unramified quasi-split groups, see \cite[Chapitre I\!I]{La13}.
\end{remark}

\subsection{The basic function}\label{sec:basicfun}
Fix a split connected reductive $F$-group $G$ together with a Borel pair $(B,T)$ over $F$. Assume that we are given
\begin{itemize}
  \item a short exact sequence
    $$ 1 \to G_0 \to G \xrightarrow{\det_G} \Gm \to 1 $$
    where $G_0$ is a split semisimple $F$-group, and on the dual side we have $\C^\times \hookrightarrow \hat{G}$;
  \item the induced short exact sequence
    $$ 0 \to X_*(T_0) \to X_*(T) \xrightarrow{\det_G} \Z \to 0 $$
    where $\Z$ is identified with $X_*(\Gm)$ by $k \mapsto [z \mapsto z^k]$;
  \item a representation $\rho: \hat{G} \to \GL(V, \C)$ such that $\rho(z) = z\cdot\identity$ for every $z \in \C^\times$.
\end{itemize}
Such data can be obtained systematically from the recipe in \S\ref{sec:Ngo}; in practice we have to allow more general situations, such as the case of reducible $(\rho, V)$.

Choose a hyperspecial vertex in the Bruhat-Tits building of $G$ which lies in the apartment determined by $T$. The corresponding hyperspecial subgroup of $G(F)$ is denoted by $K$ as usual. Define the Satake isomorphism $\mathcal{S}$ accordingly.

Let $c \in \hat{T}/W$ and $\pi_c$ be the $K$-unramified irreducible representation with Satake parameter $c$. Our starting point is the formula \eqref{eqn:L-Sym} for the $L$-factor
$$ L(\pi_c, \rho, X) = \sum_{k \geq 0} \Tr\left( \Sym^k\rho(c) \right) X^k \quad \in \C\llbracket X\rrbracket $$
where $X$ is an indeterminate.  Also observe that for all $s \in \C$,
$$ L(\pi_c \otimes |\det_G|_F^s, \rho, X) = L(\pi_c, \rho, q_F^{-s} X). $$

For every $k \geq 0$ and $\lambda \in X_*(T)_-$,
\begin{itemize}
  \item $V(\lambda)$ denotes the irreducible representation of $\hat{G}$ with highest weight $\lambda \in X^*(\hat{T})_- = X_*(T)_-$ relative to $B^-$;
  \item $\mult(\Sym^k \rho : V(\lambda)) \in \Z_{\geq 0}$ denotes the multiplicity of $V(\lambda)$ in $\Sym^k \rho$.
\end{itemize}

Hence
$$ L(\pi_c, \rho, X) = \sum_{k \geq 0} \sum_{\lambda \in X_*(T)_-} \mult(\Sym^k \rho : V(\lambda)) \Tr(V(\lambda))(c) X^k . $$

By the Kato-Lusztig formula (Theorem \ref{prop:Kato-Lusztig}), it equals
\begin{multline*}
  \sum_{k \geq 0} \left( \sum_{\substack{\mu, \lambda \in X_*(T)_- \\ \mu \leq \lambda}} \mult(\Sym^k \rho : V(\lambda)) q_F^{-\angles{\rho_{B^-}, \mu}} K_{\lambda, \mu}(q_F^{-1}) \mathcal{S}(\mathbbm{1}_{K\mu(\varpi)K})(c) \right) X^k \\
  = \sum_{\mu \in X_*(T)_-} \left( \sum_{k \geq 0} \sum_{\substack{\lambda \in X_*(T)_- \\ \lambda \geq \mu}} K_{\lambda,\mu}(q_F^{-1}) \mult(\Sym^k \rho : V(\lambda)) X^k \right) q_F^{-\angles{\rho_{B^-}, \mu}} \mathcal{S}(\mathbbm{1}_{K\mu(\varpi)K})(c).
\end{multline*}

At this stage, one has to observe that each weight $\nu$ of $\Sym^k \rho$ satisfies $\det_G \nu = k$. Thus for each $\mu \in X_*(T)_-$, the inner sum can be taken over $k = \det_G(\mu)$. Our manipulations are thus justified in $\C\llbracket X\rrbracket$.

Introduce now another indeterminate $q$. For $\mu \in X_*(T)_-$, we set
\begin{gather}\label{eqn:c_mu}
  c_\mu(q) := \begin{cases}\displaystyle
    \sum_{\substack{\lambda \in X_*(T)_- \\ \lambda \geq \mu}} K_{\lambda,\mu}(q^{-1}) \mult(\Sym^{\det_G(\mu)} \rho : V(\lambda)), & \text{if } \det_G(\mu) \geq 0, \\
    0, & \text{otherwise}.
  \end{cases}
\end{gather}

We have to justify the rearrangement of sums. Given $\mu$ with $\det_G(\mu) = k \geq 0$, the expression \eqref{eqn:c_mu} is a finite sum over those $\lambda$ with $\det_G(\lambda)=k$ as explained above, thus is well-defined. On the other hand, given $k \geq 0$, there are only finitely many $V(\lambda)$ that appear in $\Sym^k \rho$, thus only finitely many $\mu \in X_*(T)_-$ with $\det_G(\mu)=k$ and $c_\mu(q) \neq 0$. To sum up, we arrive at the following equation in $\C\llbracket X \rrbracket$
\begin{gather}\label{eqn:L-c_mu}
  L(\pi_c, \rho, X) = \sum_{\mu \in X_*(T)_-} c_\mu(q_F) q_F^{-\angles{\rho_{B^-}, \mu}} \mathcal{S}(\mathbbm{1}_{K\mu(\varpi)K})(c) \cdot X^{\det_G(\mu)}.
\end{gather}

Define the function $\varphi_{\rho,X}: T(F)\sslash K_T \to \C[X]$ by
$$ \varphi_{\rho,X} = \sum_{\mu \in X_*(T)_-} c_\mu(q_F) q_F^{-\angles{\rho_{B^-}, \mu}} \mathcal{S}(\mathbbm{1}_{K\mu(\varpi)K}) X^{\det_G(\mu)}. $$

The preceding discussion actually showed that $\varphi_{\rho,X} \in \mathcal{H}_\text{ac}(T(F)\sslash K_T; \C[X])^W$. It is meaningful to evaluate $\varphi_{\rho,X}$ at $c \in \hat{T}/W$ by the sum \eqref{eqn:L-c_mu}: it converges in the $X$-adic topology.

\begin{definition}\label{def:basic}
  Define the \textit{basic function} $f_{\rho,X} \in \mathcal{H}_{\text{ac}}(G(F)\sslash K; \C[X])$ as
  $$ f_{\rho,X} := \sum_{\mu \in X_*(T)_-} c_\mu(q_F) q_F^{-\angles{\rho_{B^-}, \mu}} \mathbbm{1}_{K\mu(\varpi)K} \cdot  X^{\det_G(\mu)}. $$

  One may specialize the variable $X$. Define $f_\rho, f_{\rho, s} \in \mathcal{H}_{\text{ac}}(G(F)\sslash K)$ as the specialization at $X=1$ and $X=q_F^{-s}$ ($s \in \C$), respectively. Then
  $$ f_{\rho,s} = f_{\rho} \cdot |\det_G|_F^{s} . $$
\end{definition}

\begin{remark}
  The basic functions $f_{\rho,X}$, $f_{\rho,s}$ are never compactly supported on $G(F)$.
\end{remark}

\begin{proposition}\label{prop:basic-L}
  We have $\mathcal{S}(f_{\rho,X}) = \varphi_{\rho,X}$. Let $c \in \hat{T}/W$ and $\pi_c$ be the $K$-unramified irreducible representation with Satake parameter $c$. Let $V_c$ denote the underlying $\C$-vector space of $\pi_c$, then
  $$ \Tr(f_{\rho,X}|V_c) = \varphi_{\rho,X}(c) = L(\pi_c, \rho, X). $$

  Similarly, for $\Re(s)$ sufficiently large with respect to $c$, the operator $\pi_c(f_{\rho,s}): V_c \to V_c$ and its trace will be well-defined and
  $$ \Tr(f_{\rho,s}|V_c) = L(s, \pi_c, \rho). $$
\end{proposition}
\begin{proof}
  The first equality has been noted. As for the second equality, let us show the absolute convergence of
  \begin{gather}\label{eqn:L-conv}
    \sum_{\mu \in X_*(T)_-} c_\mu(q_F) q_F^{-\angles{\rho_{B^-}, \mu}} \Tr(\pi_c(\mathbbm{1}_{K\mu(\varpi)K})) \cdot q_F^{-\Re(s)\det_G(\mu)}
  \end{gather}
  for $\Re(s) \gg 0$. Granting this, the equalities $\Tr(f_{\rho,X}|V_c) = \varphi_{\rho,X}(c) = L(\pi_c, \rho, X)$ will follow at once (say from \eqref{eqn:L-c_mu}), in which every term is well-defined. We have to cite some results as follows.
  \begin{enumerate}
    \item Macdonald's formula for \textit{zonal spherical functions} \cite[Theorem 5.6.1]{HKP10} says that for $\mu \in X_*(T)_-$, the trace $\Tr(\pi_c(\mathbbm{1}_{K\mu(\varpi)K}))$ is equal to
      $$ \frac{q_F^{\angles{\rho_{B^-}, \mu}}}{W_\mu(q_F^{-1})} \sum_{w \in W} \prod_{\alpha \in \Sigma_{B^-}} \dfrac{1 - q_F^{-1}(w\chi_c)(\alpha^\vee(\varpi)^{-1})}{1 - (w\chi_c)(\alpha^\vee(\varpi)^{-1})} \cdot w\chi_c(\mu(\varpi)) $$
      where $\chi_c: T(F) \to \C^\times$ is a unramified character associated to $c$ and 
      $$ W_\mu(q) := \sum_{w \in W: w\mu = \mu} q^{\ell(w)}. $$
      As a function in $\mu$, the trace is thus dominated by
      $$ q_F^{\angles{\rho_{B^-}, \mu}} \max_{w \in W}\left\lvert w\chi_c(\mu(\varpi))\right\rvert. $$
    \item In Lemma \ref{prop:Ehrhart} we will see that the number of points in $\mathcal{C}_\rho \cap X_*(T)_-$ (see Corollary \ref{prop:basic-supp}) with $\det_G = k$ is of polynomial growth in $k$.
    \item The Lemma \ref{prop:c_mu-growth} asserts that $\mu \mapsto c_\mu(q_F)$ is of at most polynomial growth.
  \end{enumerate}

  Let $\gamma_c \in X^*(T)_\R$ be the element such that $q_F^{\angles{\gamma_c, \nu}} = |\chi_c(\nu(\varpi))|$ for all $\nu \in X_*(T)$. The three facts above imply the absolute convergence of \eqref{eqn:L-conv} whenever
  $$ \Re(s) > \max \left\{ \angles{w\gamma_c, \mu} : w \in W, \mu \in \mathcal{C}_\rho \cap X_*(T)_- \text{ satisfying } \det_G(\mu)=1 \right\} $$
  holds.
\end{proof}

\begin{remark}\label{rem:basic-L}
  Choose $K$-fixed vectors $v \in V_c$ and $\check{v} \in \check{V}_c$ (the contragredient representation) such that $\angles{\check{v},v}=1$. Then $\Tr(\pi_c(\mathbbm{1}_{K\mu(\varpi)K}))$ equals $\int_{K\mu(\varpi)K} \angles{\check{v}, \pi_c(x) v} \dd x$. The absolute convergence of \eqref{eqn:L-conv} is equivalent to that $f_{\rho,s}(\cdot) \angles{\check{v}, \pi_c(\cdot) v} \in L^1(G(F))$, in which case that integral equals $L(s, \pi_c, \rho)$.

  When $\pi_c$ is unitary, there is another way to control $\Re(s)$. In fact it suffices that $f_{\rho,s} \in L^1(G(F))$ since $\angles{\check{v}, \pi_c(\cdot)v}$ is uniformly bounded. In Proposition \ref{prop:L1} we will obtain a lower bound for this purpose.
\end{remark}

\subsection{Trivial estimates}\label{sec:estimates}
We will give some estimates on the coefficients $c_\mu(q_F)$ (see \eqref{eqn:c_mu}) of the basic function $f_\rho$. These estimates are called trivial since they are obtained by passing to the classical limit $q=1$.

Observe that $\det_G(\nu)=1$ for every $\nu \in \Supp(V)$, the set with multiplicities of the $T$-weights of $V$.

\begin{proposition}\label{prop:classical}
  For every $\mu \in X_*(T)_-$ we have
  $$ c_\mu(q_F) \leq c_\mu(1) = \left\lvert \left\{ (a_\nu)_\nu \in (\Z_{\geq 0})^{\Supp(V)} : \sum_{\nu \in \Supp(V)} a_\nu \nu = \mu \right\} \right\rvert . $$
\end{proposition}
\begin{proof}
  The first inequality follows \eqref{eqn:c_mu} and the well-known non-negativity of the coefficients of $K_{\lambda,\mu}(q)$. For the second equality, we use the fact that
  $$ K_{\lambda,\mu}(1) = \text{the multiplicity of $\mu$ in $V(\lambda)|_{\hat{T}}$}. $$
  This is also well-known; in fact, \eqref{eqn:K} reduces to Kostant's multiplicity formula for $V(\lambda)$ at $q=1$. Thus
  \begin{align*}
    c_\mu(1) & = \sum_{\lambda \in X^*(T)_-} \text{mult}(V(\lambda)|_{\hat{T}} : \mu) \cdot \text{mult}(\Sym^{\det_G(\mu)} \rho : V(\lambda)) \\
    & = \text{mult}(\Sym^{\det_G(\mu)}\rho|_{\hat{T}} : \mu) = \text{mult}(\Sym(\rho)|_{\hat{T}} : \mu),
  \end{align*}
  and the assertion follows.
\end{proof}

\begin{corollary}\label{prop:basic-supp}
  Let $\mathcal{C}_\rho$ be the convex hull of $\Supp(V) \sqcup \{0\}$ in $X_*(T)_\R$. Then
  \begin{enumerate}
    \item $\mathcal{C}_\rho$ is a strongly convex polyhedral cone, that is, it contains no lines;
    \item the basic function $f_{\rho,X}$ is supported in $K(\mathcal{C}_\rho \cap X_*(T)_-)K$, where we embed $X_*(T)_-$ into $T(F)$ by $\mu \mapsto \mu(\varpi)$.
  \end{enumerate}
\end{corollary}
\begin{proof}
  The first assertion is evident and the second follows from Proposition \ref{prop:classical}.
\end{proof}

\begin{corollary}\label{prop:c_mu-growth}
  There is a polynomial function $Q: X_*(T)_\R \to \R$ depending solely on $(\rho,V)$ such that $c_\mu(q_F) \leq |Q(\mu)|$.
\end{corollary}
\begin{proof}
  Combine Proposition \ref{prop:classical} with the fact the number of non-negative integer solutions of $\sum_{\nu \in \Supp(V)} a_\nu \nu = \mu$ is of polynomial growth in $\mu$.
\end{proof}

\begin{lemma}\label{prop:Ehrhart}
  There exists a polynomial function $R: \R \to \R$ depending solely on $(\rho,V)$ such that
  $$ \left\lvert \left\{ \mu \in \mathcal{C}_\rho \cap X_*(T)_- : \det_G(\mu)=k \right\}\right\rvert \leq R(k) $$
  for every $k \in \Z_{\geq 0}$.
\end{lemma}
\begin{proof}
  Observe that $\mathcal{C}_\rho \cap X_*(T)_{\R,-}$ is a convex polyhedral cone lying in the half-space $\det_G > 0$. Its intersection with $\det_G = 1$ can be viewed as an integral polytope in $X_*(T_0)_\R$. Denote this polytope by $\mathcal{P}$. Then the cardinality of $\mathcal{C}_\rho \cap X_*(T)_- \cap (\det_G = k)$ can be bounded by that of $(k\mathcal{P}) \cap X_*(T_0)$. The behaviour of $k \mapsto |k\mathcal{P} \cap X_*(T_0)|$ is described by the Ehrhart polynomial of $\mathcal{P}$: see for example \cite[6.E]{BG09}.
\end{proof}

Hereafter we choose the Haar measure on $T(F)$ such that $\mes(T(F) \cap K)=1$. Recall the integration formula
\begin{gather}\label{eqn:KAK-int}
  \int_{G(F)} f(x) \dd x = \int_K \int _K \int_{T(F)_+} f(k_1 t k_2) D_T(t) \dd t \dd k_1 \dd k_2
\end{gather}
for every measurable function $f$ on $G(F)$, where $D_T(t) := \mes(KtK)$. By \cite[I.1 (5)]{Wa03}, there exist constants $c_1, c_2 > 0$ such that
\begin{gather}\label{eqn:D_T}
  c_1 \delta_{B^-}(t)^{-1} \leq D_T(t) \leq c_2 \delta_{B^-}(t)^{-1}, \quad t \in T(F)_+
\end{gather}
where $T(F)_+ \subset T(F)$ is the image of $X_*(T)_-$ under $\mu \mapsto \mu(\varpi)$.

\begin{proposition}\label{prop:Lp}
  Let $p \geq 2$ and $s \in \C$. Then $f_{\rho,s} \in L^p(G(F))$ whenever $\Re(s) > 0$.
\end{proposition}
\begin{proof}
  Since $\delta_{B^-}(\mu(\varpi)) = q_F^{-\angles{2\rho_{B^-}, \mu}}$, we can write
  \begin{align*}
    \int_{G(F)} |f_{\rho, s}(x)|^p \dd x & = \int_{T(F)_+} |f_{\rho, s}(t)|^p D_T(t) \dd t \\
    & = \sum_{\mu \in X_*(T)_-} |f_{\rho,s}(\mu(\varpi))|^p D_T(\mu(\varpi)) \\
    & \leq c_2 \sum_{\mu \in X_*(T)_-} c_\mu(q_F)^p q_F^{-p\angles{\rho_{B^-}, \mu} - \Re(s) p \det_G(\mu)} q_F^{\angles{2\rho_{B^-}, \mu}}
  \end{align*}
  using \eqref{eqn:D_T}. As $p \geq 2$ and $\mu \in X_*(T)_-$, the exponent of $q_F$ is
  $$ (2-p)\angles{\rho_{B^-}, \mu} - \Re(s)p \det_G(\mu) \leq -\Re(s) p \det_G(\mu). $$

  On the other hand, we have seen that $c_\mu(q_F)$ is of at most polynomial growth in $\mu$. Thus so is $c_\mu(q_F)^p$ and we can drop it in the study of convergence issues, say upon replacing $s$ by $s-\epsilon$ for an arbitrarily small $\epsilon > 0$. Furthermore, the Lemma \ref{prop:Ehrhart} reduces our problem to the series $\sum_{k \geq 0} q_F^{-\Re(s) pk}$, again upon replacing $s$ by $s-\epsilon$. The latter series converges absolutely as $\Re(s) > 0$.
\end{proof}

\begin{proposition}\label{prop:L1}
  Let $s \in \C$. Then $f_{\rho,s} \in L^1(G(F))$ whenever
  $$ \Re(s) > m := \max_{\xi \in \Supp(V)} \angles{\rho_{B^-}, \xi}. $$

  If it is the case, then $\int_{G(F)} f_{\rho,s}(x) \dd x = L(s, \mathbbm{1}, \rho)$ where $\mathbbm{1}$ denotes the trivial representation of $G(F)$.
\end{proposition}
\begin{proof}
  Let $\epsilon > 0$ be arbitrarily small. As in the proof of Proposition \ref{prop:Lp} (now with $p=1$), formula \eqref{eqn:KAK-int} gives
  $$ \int_{G(F)} |f_{\rho, s}(x)| \dd x \leq c_2 \sum_{\mu \in X_*(T)_-} c_\mu(q_F) q_F^{\angles{\rho_{B^-}, \mu} - \Re(s) \det_G(\mu)}. $$

  Again, by Corollary \ref{prop:c_mu-growth} we may discard the term $c_\mu(q_F)$ upon replacing $s$ by $s - \epsilon/2$. For every $\mu \in \mathcal{C}_\rho \cap X_*(T)_-$ we have $\angles{\rho_{B^-}, \mu} \leq \det_G(\mu)m$. We may collect terms according to $k = \det_G(\mu)$ by Lemma \ref{prop:Ehrhart}, at the cost of replacing $s-\epsilon/2$ by $s-\epsilon$. It remains to observe that
  $$ \sum_{k=0}^\infty q_F^{k (m - \Re(s) + \epsilon)} = \left( 1 - q_F^{m - \Re(s) + \epsilon} \right)^{-1} $$
  whenever $\Re(s) - \epsilon > m$.

  Suppose $\Re(s) > m$ so that $f_{\rho, s}$ is integrable. The second assertion follows from the Remark \ref{rem:basic-L} applied to the trivial representation $\mathbbm{1}$.
\end{proof}

The next result concerns Harish-Chandra's Schwartz space $\mathcal{C}(G)$, which is a strict inductive limit of Fréchet spaces, cf. \cite[III.6]{Wa03}. We refer to \cite{Wa03} for the definition of the height function $\sigma$ and Harish-Chandra's $\Xi$-function, etc.

\begin{proposition}\label{prop:tempered}
  For every $s$ with $\Re(s) \geq 0$, the function $f_{\rho,s}$ defines a tempered distribution on $G(F)$ in the sense that the linear functional
  \begin{align*}
    \mathcal{C}(G) & \longrightarrow \C, \\
    h & \longmapsto \int_{G(F)} h(x)f_{\rho,s}(x) \dd x
  \end{align*}
  on Harish-Chandra's Schwartz space $\mathcal{C}(G)$ is well-defined and continuous.
\end{proposition}
\begin{proof}
  Since $f_{\rho,s}(x) \neq 0$ implies $|\det_G(x)|_F \leq 1$, it suffices to treat the case $s=0$. Recall that $q_F^{-\angles{\rho_{B^-}, \mu}} = \delta_{B^-}(\mu(\varpi))^{1/2}$ for $\mu \in X_*(T)_-$. By the Definition \ref{def:basic} of $f_\rho$, the Lemma \ref{prop:c_mu-growth} and \cite[Lemme II.1.1]{Wa03}, there exist $r \in \R$ and $c > 0$ such that
  $$ 0 \leq f_\rho(x) \leq c(1+\sigma(x))^r \Xi(x), \quad x \in T(F)_+. $$

  Since $\Xi$ and $\sigma$ are both bi-invariant under $K$, the same estimates holds for all $x \in G(F)$. Therefore $f_\rho$ belongs to the space $C^w_\text{lisse}(G)$ of \cite[III.2]{Wa03}. Now the assertion follows from the discussion in \cite[p.273]{Wa03}.
\end{proof}

\section{Generalized Kostka-Foulkes polynomials}\label{sec:gKF}
Let $\Bbbk$ be an algebraically closed field of characteristic zero. We fix a connected reductive $\Bbbk$-group $G$, together with a chosen Borel pair $(B,T)$ for $G$. We will write the Levi decomposition as $B=TU$, $\mathfrak{b} = \mathfrak{t} \oplus \mathfrak{u}$. Accordingly, we have the Weyl group $W$, the weight lattice $X^*(T)$ and the monoid $X^*(T)_-$ of anti-dominant weights relative to $B$, etc.

These objects on the dual side should not be confused with those in the previous sections.

\subsection{Combinatorial setup}\label{sec:com-setup}
The main reference here is \cite{Pan10}, nonetheless we
\begin{inparaenum}[(i)]
  \item consider reductive groups and
  \item work with anti-dominant weights.
\end{inparaenum}

Consider the data
\begin{itemize}
  \item $P$: a parabolic subgroup of $G$ containing $B$, whose unipotent radical we denote by $U_P$;
  \item $\Psi$: a set \textit{with multiplicities} of weights in $X^*(T)$, such that
    \begin{itemize}
      \item $\Psi$ lies in a strongly convex cone in $X^*(T)_\R$, that is, a convex cone containing no lines,
      \item $\Psi$ is the set with multiplicities of $T$-weights of a $P$-stable subspace $N$ of a finite-dimensional representation $W$ of $G$.
    \end{itemize}
\end{itemize}

Under these assumptions, we may define $\mathcal{P}_\Psi(\nu; q) \in \Z[q]$, for each $\nu \in X^*(T)$, by requiring
\begin{gather}\label{eqn:P-Psi}
  \prod_{\alpha \in \Psi} (1 - qe^{-\alpha}) = \sum_{\nu \in X^*(T)} \mathcal{P}_\Psi(\nu; q) e^\nu.
\end{gather}

\begin{definition}\label{def:gKF}
  Let $\lambda \in X^*(T)_-$ and $\mu \in X^*(T)$, define the corresponding generalized Kostka-Foulkes polynomial as
  $$ m^\mu_{\lambda, \Psi}(q) := \sum_{w \in W} (-1)^{\ell(w)} \mathcal{P}_\Psi \left( w(\lambda + \rho_{B^-}) - (\mu + \rho_{B^-}); q \right) \quad \in \Z[q]. $$
\end{definition}

In the examples below we take $P=B$.

\begin{example}\label{ex:q-analogue}
  Consider the special case in which $W := \mathfrak{g}$ is the adjoint representation $\Ad$ of $G$, and $N := \mathfrak{u}$ is a $B$-stable subspace. Thus $\Psi = \Sigma_B$ is the set of positive roots relative to $B$; it certainly lies in a strongly convex cone. We obtain
  $$ \prod_{\alpha \in \Sigma_{B^-}} (1 - q e^\alpha)^{-1} = \prod_{\alpha \in \Sigma_B} (1 - qe^{-\alpha})^{-1} = \sum_\nu \mathcal{P}_\Psi(\nu; q) e^\nu. $$
  Hence $\mathcal{P}_\Psi(\nu; q)$ equals the function $\mathcal{P}(\nu; q)$ introduced in \S\ref{sec:Satake-inversion}, defined on the root datum dual to that of $(G, T)$ and relative to $B^-$. Consequently we recover Lusztig's $q$-analogue \eqref{eqn:K}, namely
  $$ m^\mu_{\lambda, \Psi}(q) = K_{\lambda, \mu}(q), $$
  in the anti-dominant setting, which is precisely our choice throughout \S\ref{sec:L-basic}.

  We record some well-known properties of $K_{\lambda, \mu}(q)$, for $\lambda, \mu \in X^*(T)_-$.
  \begin{enumerate}
    \item $K_{\lambda, \lambda} = 1$.
    \item $K_{\lambda, \mu} \neq 0$ only if $\mu \leq \lambda$ for the Bruhat order relative to $B^-$.
    \item $\deg_q K_{\lambda, \mu} = \angles{\lambda-\mu, \check{\rho}_{B^-}}$ whenever $\mu \leq \lambda$.
    \item $K_{\lambda,\mu}(1)$ equals the weight multiplicity $\dim_\Bbbk V(\lambda)_\mu$ of the irreducible representation $V(\lambda)$ of highest weight $\lambda$.
  \end{enumerate}

  The first two properties are actually shared by all generalized Kostka-Foulkes polynomials $m^\mu_{\lambda, \Psi}(q)$ with $\Psi \subset \Sigma_B$; see \cite[Lemma 2.3]{Pan10}.
\end{example}

\begin{example}\label{ex:basic-fcn}
  Assume that we are given a short exact sequence $1 \to \Gm \to G \to G_0 \to 1$ with $G_0$ semisimple. It induces a homomorphism $X^*(T) \to X^*(\Gm) = \Z$ denoted as $\det_G$. Let $(\rho, V)$ be a representation of $G$ such that $\rho(z) = z\cdot\identity$ for all $z \in \Gm$. Now take $N := V \oplus \mathfrak{u}$, a $B$-stable subspace of some representation of $G$ (where $\mathfrak{u}$ is acted upon by $\Ad|_B$). We have
  $$ \Psi = \Supp(V) \sqcup \Sigma_B $$
  where $\Supp(V)$ is the set with multiplicities of $T$-weights, as usual.  The elements of $\Supp(V)$ lie on $\det_G = 1$, whereas the elements of $\Sigma_B$ lie in a chamber in the subspace $\det_G = 0$. Evidently, $\Psi$ is contained in a strongly convex cone in $X^*(T)_\R$. Hence the polynomials $m^\mu_{\lambda, \Psi}(q)$ are well-defined.

  Notice that in this case, $m^\mu_{0, \Psi} \neq 0$ only when $\det_G(\mu) \geq 0$.
\end{example}

\subsection{Geometric setup}
We refer to \S\ref{sec:intro} for the formalism of $G$-varieties and $G$-linearized sheaves.

Let $H$ be a $\Bbbk$-subgroup of $G$. Let $Z$ be an $H$-variety, we define the associated \textit{homogeneous fiber space} as
$$ G \utimes{H} Z := (G \times Z) \bigg/ ((gh,z) \sim (g, hz), \forall h \in H). $$
The quotient here is to be understood as the \textit{geometric quotient} $(G \times Z)/H$, where $h \in H$ acts by $h(g,z) = (gh^{-1}, hz)$. In order to form the quotient as a $\Bbbk$-variety we should assume that $Z$ can be covered by $H$-stable quasi-projective open subsets, which is possible under mild conditions, eg. whenever $Z$ is quasi-projective. We refer to \cite[\S 2]{Ti11} or \cite[II. \S 4]{AGIV} for details.

In what follows we assume implicitly that $Z$ has all the required properties to ensure the existence of $G \utimes{H} Z$ as a $\Bbbk$-variety. The $G$-action on $G \utimes{H} Z$ is descended from the left translation on the first component of $G \times Z$.  The natural projection $G \utimes{H} Z \to G/H$ descended from $\text{pr}_1: G \times Z \to G$ is locally trivial in the étale topology, with fibers isomorphic to $Z$.  A typical element in $G \utimes{H} Z$ is denoted by $[g,z]$, signifying the image of $(g,z) \in G \times H$ under the quotient morphism.

\begin{remark}\label{rem:locally-trivial}
  The bundles $G \utimes{H} Z \to G/H$ are locally trivial in the Zariski topology for all $Z$ if and only if so is $G \to G/H$. The latter condition holds when $H$ is a parabolic subgroup of $G$, due to Bruhat decomposition. See the discussion in \cite[II. \S 4.8]{AGIV}.
\end{remark}

We will make use of the following avatar of projection formula: suppose that $Z$ is actually endowed with a $G$-variety structure, and let $Z'$ be any $H$-variety, then there is an isomorphism of $G$-varieties
\begin{equation}\label{eqn:proj-formula} \begin{split}
  G \utimes{H} (Z' \times Z) & \rightiso (G \utimes{H} Z') \times Z \\
  [g, (z',z)] & \mapsto ([g,z'], gz).
\end{split}\end{equation}
Here $Z' \times Z$ is equipped with the diagonal $H$-action, and $(G \utimes{H} Z') \times Z$ is equipped with the diagonal $G$-action. In particular $G \utimes{H} Z \simeq G/H \times Z$.

Next, we consider the ``collapsing'' of homogeneous fiber spaces \cite{Kem76}. Let $P$ be a parabolic subgroup of $G$. Let $Z$ be a closed $G$-variety and $W$ be a $P$-subvariety of $Z$. We have the commutative diagram of $G$-varieties
$$\xymatrix{
  G \utimes{P} W \ar[r] \ar[rd]_{\pi} & G \utimes{P} Z \ar[r]^{\sim} & G/P \times Z \ar[ld]^{\text{pr}_2} \\
  & Z &
}$$
where $\pi$ sends $[g,w]$ to $gw$. Also note that $G \utimes{P} W \to G \utimes{P} Z$ is a closed immersion.

\begin{lemma}\label{prop:collapsing}
  For $Z$ and $W$ as above,
  \begin{enumerate}
    \item $\pi: G \utimes{P} W \to Z$ is a projective morphism;
    \item the set-theoretic image of $\pi$ equals $G \cdot W$ and is a closed subvariety of $Z$.
  \end{enumerate}
\end{lemma}
\begin{proof}
  Note that $\text{pr}_2: G/P \times Z \to Z$ is projective since $G/P$ is a projective $\Bbbk$-variety. Hence $\pi$ is projective as well, by the diagram above. Since projective morphisms have closed images, the second assertion follows.
\end{proof}

For any finite-dimensional representation $W$ of $P$, the construction above yields a $G$-equivariant vector bundle $p: G \utimes{P} W \to G/P$; see Remark \ref{rem:locally-trivial}. We denote by $\mathscr{L}_{G/P}(W)$ the locally free $\mathcal{O}_{G/P}$-module of its sections. More precisely, for every open subset $U \subset G/P$ we have
$$ \Gamma(U, \mathscr{L}_{G/P}(W)) = \left\{ s: p^{-1}U \xrightarrow{\text{morphism}} W, \quad \forall h \in P, \; s(gh^{-1}) = hs(g) \right\}. $$

This $\mathcal{O}_{G/P}$-module is canonically $G$-linearized. The natural $G$-action on $\Gamma(G/P, \mathscr{L}_{G/P}(W))$ is $gs: x \mapsto s(g^{-1}x)$. More generally, $G$ acts on the cohomology groups $H^i(G/P, \mathscr{L}_{G/P}(W))$. We have a canonical isomorphism $\mathscr{L}_{G/P}(W^\vee) \simeq \mathscr{L}_{G/P}(W)^\vee$.

We record the following standard facts.
\begin{proposition}\label{prop:bundle-pushforward}
  Let $p: V \to X$ be a vector bundle with sheaf of sections $\mathcal{F}$.
  \begin{enumerate}
    \item For every locally free $\mathcal{O}_X$-module $\mathscr{L}$ of finite rank, we have $p_* p^* \mathscr{L} \simeq \mathscr{L} \otimes \Sym(\mathcal{F}^\vee) = \bigoplus_{k \geq 0} \left( \mathscr{L} \otimes \Sym^k(\mathcal{F}^\vee) \right)$.
    \item For every quasi-coherent $\mathcal{O}_V$-module $\mathcal{G}$, we have $H^i(V, \mathcal{G}) \simeq H^i(X, p_* \mathcal{G})$ for all $i$.

    These isomorphisms are all canonical. Recall another elementary fact:
    \item Cohomology on noetherian schemes commutes with direct limits. 
  \end{enumerate}
\end{proposition}
\begin{proof}
  The first isomorphism is an easy consequence of the projection formula and the identification $p_* \mathcal{O}_V = \Sym(\mathcal{F}^\vee)$. The second isomorphism results from the degeneration for the Leray spectral sequence for the affine morphism $p: V \to X$.
\end{proof}

Consider a parabolic subgroup $P$ containing $B$ with unipotent radical $U_P$. The inclusion $B \hookrightarrow P$ induces a surjective homomorphism $B/B_\text{der}U = B/U \twoheadrightarrow P/P_\text{der} U_P$ between $\Bbbk$-tori. Hence
$$ X^*(P) := \Hom(P, \Gm) \hookrightarrow \Hom(B, \Gm) = X^*(T) $$
where the Hom-groups are taken in the category of $\Bbbk$-groups. For $\mu \in X^*(P)$, we write $\Bbbk_\mu$ for the corresponding one-dimensional representation of $P$, and use the shorthand
$$ \mathscr{L}_{G/P}(\mu) := \mathscr{L}_{G/P}(\Bbbk_\mu), $$
thus $\mathscr{L}_{G/P}(-\mu) \simeq \mathscr{L}_{G/P}(\mu)^\vee$ canonically. We need the following vanishing theorem due to Panyushev.

\begin{theorem}[{\cite[Theorem 3.1]{Pan10}}]\label{prop:vanishing}
  Let $N$ be a representation of $P$ satisfying the conditions in \S\ref{sec:com-setup}; in particular we have the homogeneous fibration $p: G \utimes{P} N \to G/P$ and a collapsing morphism $\pi: G \utimes{P} N \to G \cdot N \subset W$. If $\pi$ is generically finite, then
  $$ H^i \left( G \utimes{P} N, \; p^* \mathscr{L}_{G/P}(\mu) \right) = 0 \quad \text{ for all } i \geq 1, $$
  whenever
  \begin{gather}\label{eqn:mu-vanishing-condition}
    \mu + \sum_{\alpha \in \Psi} \alpha - \sum_{\alpha \in \Sigma_P} \alpha \in X^*(T)_-. 
  \end{gather}
  Here $\Psi$ is the set with multiplicities of the $T$-weights of $N$, and $\Sigma_P$ is the set of roots in $\mathfrak{u}_P$.
\end{theorem}
\begin{proof}
  We reproduce the proof in \cite{Pan10} under our setup. Let $\nu \in X^*(T)_+$ and set $\mathbf{Z} := G \utimes{P} N$ and $\mathbf{U} := G \utimes{P} (N \times \Bbbk_\nu)$. Write $\mathscr{L}_{\mathbf{Z}}(\cdots) := p^* \mathscr{L}_{G/P}(\cdots)$; the same for $\mathscr{L}_{\mathbf{U}}(\cdots)$. The natural projection $\eta: \mathbf{U} \to \mathbf{Z}$ makes $\mathbf{U}$ into the total space of a line bundle over $\mathbf{Z}$ whose sheaf of sections is $\mathscr{L}_{\mathbf{Z}}(\nu)$. Let
  $$ \gamma := - \sum_{\alpha \in \Psi} \alpha + \sum_{\alpha \in \Sigma_P} \alpha . $$

  Note that $\gamma \in X^*(P)$. Let $\omega_{\mathbf{U}}$ (resp. $\omega_{G/P}$) denote the dualizing sheaf on $\mathbf{U}$ (resp. $G/P$). Using the smooth fibration $\mathbf{U} \to G/P$ we see that $\omega_{\mathbf{U}}$ equals the tensor product of
  $$ \mathscr{L}_{\mathbf{U}} \left( \det N^\vee \otimes \Bbbk_{-\nu} \right), $$
  the sheaf of relative differentials of top degree, with the pullback of $\omega_{G/P}$ to $\mathbf{U}$. However, $\omega_{G/P} = \mathscr{L}_{G/P}(\det \mathfrak{u}_P)$ as $(\mathfrak{g}/\mathfrak{p})^\vee = \mathfrak{p}^\perp \simeq \mathfrak{u}_P$ via the Killing form of $\mathfrak{g}_\text{der}$. Thus we obtain
  $$ \omega_{\mathbf{U}} = \mathscr{L}_{\mathbf{U}}\left( \det N^\vee \otimes \det\mathfrak{u}_P \otimes \Bbbk_{-\nu} \right) = \eta^* \mathscr{L}_{\mathbf{Z}}(\gamma - \nu). $$

  Notice that each identification is equivariant. Proposition \ref{prop:bundle-pushforward} applied to $\eta$ yields
  \begin{gather}\label{eqn:omega_U-cohomology}
      H^i(\mathbf{U}, \omega_{\mathbf{U}}) = H^i \left( \mathbf{Z}, \; \bigoplus_{k \geq 0} \mathscr{L}_{\mathbf{Z}}(\gamma - (k+1)\nu) \right).
  \end{gather}

  Let $V(\nu)$ be the irreducible representation of $G$ with highest weight $\nu$ relative to $B$. Embed $\Bbbk_\nu$ into $V(\nu)$ as a the highest weight subspace. We deduce the collapsing
  $$ \pi': \mathbf{U} = G \utimes{P} (N \times \Bbbk_\nu) \twoheadrightarrow G \cdot (N \times \Bbbk_\nu) \subset W \times V(\nu). $$
  It is proper by Lemma \ref{prop:collapsing} and generically finite since $\pi$ is. Therefore Kempf's vanishing theorem \cite[Theorem 4]{Kem76} asserts that $H^i(\mathbf{U}, \omega_{\mathbf{U}})$ vanishes for all $i \geq 1$. Taking the summand $k=0$ in \eqref{eqn:omega_U-cohomology} with $\mu := \gamma - \nu$ gives the vanishing of $H^i(\mathbf{Z}, p^* \mathscr{L}_{G/P}(\mu))$. The condition $\nu \in X^*(T)_+$ translates into \eqref{eqn:mu-vanishing-condition}.
\end{proof}

\subsection{Poincaré series}
\begin{definition}\label{def:P-series}
  Let $M$ be $\Bbbk$-vector space, possibly of infinite dimension.
  \begin{itemize}
    \item Assume $M$ is equipped with a $\Z$-grading, or equivalently with a $\Gm$-action, such that each graded piece $M_k$ ($k \in \Z$) is finite-dimensional. Assume moreover that $M_k = \{0\}$ for $k < 0$. Introduce an indeterminate $q$ and define the Poincaré series of $M$ as
      $$ \mathbf{P}(M; q) := \sum_{k \geq 0} (\dim_\Bbbk M_k) q^k \quad \in \Z\llbracket q \rrbracket. $$
    \item The multi-graded version will also be needed. Let $\mathfrak{X}$ be a commutative monoid (written additively) that is finitely generated and isomorphic to some submonoid of $\Z^r_{\geq 0}$ for some $r$. Such monoids are called \textit{positive affine monoids} in \cite[2.15]{BG09}. Assume that $M$ is equipped with a grading by $\mathfrak{X}$, such that the graded piece $M_\mu$ is finite-dimensional for every $\mu \in \mathfrak{X}$. We may define its Poincaré series
      $$ \mathbf{P}(M) := \sum_{\mu \in \mathfrak{X}} (\dim_\Bbbk M_\mu) e^\mu $$
      which belongs to the completion $\Z\llbracket \mathfrak{X} \rrbracket$ of the monoid ring $\Z[\mathfrak{X}]$ with respect to the ideal generated by non-invertible elements in $\mathfrak{X}$. The previous case corresponds to $\mathfrak{X} = \Z_{\geq 0}$.
    \item Furthermore, assume that $M$ is a representation of $G$ admitting a compatible grading by $\mathfrak{X}$ as above, such that $\dim_\Bbbk M_\mu < \infty$ for all $\mu$. For every finite-dimensional irreducible representation $V$ of $V$, let $M[V]$ be the $V$-isotypic subrepresentation of $M$. Define the formal (infinite) sum
    $$ \mathbf{P}(M) := \sum_{[V]} \mathbf{P}(M[V]) [V] $$
      where $[V]$ ranges over the equivalence classes of finite-dimensional irreducible representations.
    \item The previous case generalizes to virtual representations of $G$, by setting $\mathbf{P}(M_1 - M_2) = \mathbf{P}(M_1) - \mathbf{P}(M_2)$. 
  \end{itemize}
\end{definition}

Let $P$ be a parabolic subgroup of $G$ containing $B$, and $N$ be a finite-dimensional representation of $P$. Note that $\Gm$ acts on $N$ by dilation, which clearly commutes with the $P$-action. We deduce a $G \times \Gm$-action on $G \utimes{P} N$, namely
$$ (g,z) \cdot [x,n] = [gx, z^{-1} n], \quad g \in G, \; z \in \Gm, \; [x,n] \in G \utimes{P} N. $$

The bundle map $p: G \utimes{P} N \to G/P$ is $G \times \Gm$-equivariant if we let $\Gm$ act trivially on $G/P$. Let $\mathcal{F}$ be a $G$-linearized locally free $\mathcal{O}_{G/P}$-module of finite rank. Then $p^* \mathcal{F}$ is canonically $G \times \Gm$-linearized. All in all, we obtain a $G \times \Gm$-representation $H^i(G \utimes{P} N, p^* \mathcal{F})$ for each $i$, or equivalently, a $G$-representation with a compatible $\Z$-grading. The grading is given by the degrees along the fibers of $p$. In concrete terms, the $k$-th graded piece is identified with
$$ H^i(G/P, \mathcal{F} \otimes \Sym^k (N^\vee)) $$
for every $k \in \Z$, by Proposition \ref{prop:bundle-pushforward}.

In view of the Definition \ref{def:P-series}, we may define
\begin{gather}\label{eqn:EP}
  \chi \left( G \utimes{P} N, \; p^* \mathcal{F} \right) := \sum_{i \geq 0} (-1)^i H^i \left( G \utimes{P} N, \; p^* \mathcal{F} \right),
\end{gather}
an element of the Grothendieck group of representations of $G$ graded by $\Z_{\geq 0}$, each graded piece being a finite-dimensional representation of $G$. Thus we may consider its Poincaré series.

The following result is well-known in the case for Lusztig's $q$-analogues; see \cite{He76, He80, Bry89, Bro93}. Here we state and reprove it for reductive groups and anti-dominant weights, following the arguments in \cite[Theorem 3.8]{Pan10} for generalized Kostka-Foulkes polynomials. Our convention makes the annoying contragredients in \textit{loc. cit.} disappear.

\begin{theorem}\label{prop:EP-gKF}
  Let $N$, $P$ and $\Psi$ be as in \S\ref{sec:com-setup} and assume that $P=B$. For every $\mu \in X^*(T)_-$, we have the equality
  $$ \mathbf{P} \left( \chi \left( G \utimes{P} N, \; p^*\mathscr{L}_{G/P}(\mu) \right); q \right) = \sum_{\lambda \in X^*(T)_-} m^\mu_{\lambda, \Psi}(q) [V(\lambda)] $$
  of formal sums.
\end{theorem}

These sums may be understood $q$-adically: for each $k$ the coefficient of $q^k$ is a finite sum.

\begin{proof}
  Take $P=B$. For every $k \geq 0$ we have
  $$ \chi\left( G/B, \; \mathscr{L}_{G/B}(\Bbbk_\mu \otimes \Sym^k (N^\vee))\right) = \chi\left( G/B, \; \mathscr{L}_{G/B}(\Bbbk_\mu \otimes \Sym^k (N^\vee))^{\text{ss}} \right) $$
  in the Grothendieck group of finite-dimensional $G$-representations, where we have used an obvious variant of \eqref{eqn:EP}, and we denote by $(\cdots)^\text{ss}$ the semi-simplification of a finite-dimensional representation of $B$, namely the direct sum of its Jordan-Hölder factors. Therefore $\Sym^k (N^\vee)^{\text{ss}}$ is a sum of elements in $X^*(T)$ inflated to $B$.

  By virtue of the Proposition \ref{prop:bundle-pushforward}, the left-hand side in the assertion equals
  \begin{multline*}
    \sum_{k \geq 0}  \mathbf{P}\left( \chi\left( G/B, \; \mathscr{L}_{G/B}(\Bbbk_\mu \otimes \Sym^k (N^\vee))^{\text{ss}} \right)  \right) q^k = \\
    \sum_{k \geq 0} \sum_{\nu \in X^*(T)} \text{mult}(\Sym^k (N^\vee)|_T : \nu) \cdot \mathbf{P}\left( \chi\left( G/B, \mathscr{L}_{G/B}(\mu+\nu) \right) \right) q^k = \\
    \sum_{\nu \in X^*(T)} \mathcal{P}_\Psi(\nu; q) \cdot \mathbf{P}\left( \chi\left( G/B, \mathscr{L}_{G/B}(\mu+\nu) \right) \right), \quad \text{ by the definition \eqref{eqn:P-Psi} for } \mathcal{P}_\Psi(\cdot ; q).
  \end{multline*}

  Now we invoke the Borel-Weil-Bott theorem. It asserts that
  \begin{itemize}
    \item either $\mu + \nu + \rho_{B^-}$ does not intersect $W \cdot (X^*(T)_- + \rho_{B^-})$, in which case
      $$ H^i(G/B, \mathscr{L}_{G/B}(\mu+\nu)) = 0 \quad \text{ for all } i. $$
    \item or there exists $\lambda \in X^*(T)_-$ and $w \in W$ such that $\mu + \nu + \rho_{B^-} = w(\lambda + \rho_{B^-})$, in which case the pair $(w, \lambda)$ is unique and
      $$ H^i(G/B, \mathscr{L}_{G/B}(\mu+\nu)) \simeq \begin{cases}
	V(\lambda), & \text{ if } i = \ell(w), \\
	0, & \text{ otherwise},
      \end{cases}$$
    as $G$-representations.
  \end{itemize}

  Therefore the last sum can be rearranged according to $(w, \lambda) \in W \times X^*(T)_-$. The result is
  \begin{gather*}
    \sum_{\lambda \in X^*(T)_-} \sum_{w \in W} (-1)^{\ell(w)} \mathcal{P}_\Psi( w(\lambda + \rho_{B^-}) - (\mu + \rho_{B^-}); q) [V(\lambda)]  = \sum_{\lambda \in X^*(T)_-} m^\mu_{\lambda, \Psi}(q) [V(\lambda)]
  \end{gather*}
  as asserted.
\end{proof}

\begin{corollary}
  The assertion in Theorem \ref{prop:EP-gKF} holds for any parabolic subgroup $P$ containing $B$.
\end{corollary}
\begin{proof}
  Use the Leray spectral sequence for $G/P \to G/B$ to reduce to the case $P=B$; see \cite[Theorem 3.9]{Pan10}.
\end{proof}

\begin{corollary}\label{prop:gKF-positivity}
  If $\mu \in X^*(T)$ satisfies
  $$ H^i \left( G \utimes{P} N, \; p^* \mathscr{L}_{G/P}(\mu) \right) = 0 \quad \text{ for all } i \geq 1, $$
  then we have
  $$ \mathbf{P} \left( \Gamma \left( G \utimes{P} N, \; p^*\mathscr{L}_{G/P}(\mu) \right); q \right) = \sum_{\lambda \in X^*(T)_-} m^\mu_{\lambda, \Psi}(q) [V(\lambda)]. $$
  In particular, in this case $m^\mu_{\lambda, \Psi}(q) \in \Z_{\geq 0}[q]$ for all $\lambda \in X^*(T)_-$.
\end{corollary}
In this article we will only use the case $P=B$.

Now comes the case of Lusztig's $q$-analogues $K_{\lambda, \mu}(q)$ of the Example \ref{ex:q-analogue}. We concentrate on the case of anti-dominant $\mu$. The following result appeared first in \cite[\S 5]{Bry89}.

\begin{corollary}\label{prop:K-as-P-series}
  Let $\lambda, \mu \in X^*(T)_-$. Then $K_{\lambda, \mu}(q)$ equals the following $\Z_{\geq 0}$-graded Poincaré series of the $\Bbbk$-vector space
  $$ \Hom_G \left( V(\lambda), \Gamma\left( G \utimes{B} \mathfrak{u}, \; p^* \mathscr{L}_{G/B}(\mu) \right) \right); $$
  or equivalently, of the space of $G$-invariants
  $$ \left( V(\lambda)^\vee \otimes  \Gamma\left( G \utimes{B} \mathfrak{u}, \; p^* \mathscr{L}_{G/B}(\mu) \right) \right)^G . $$
  In both cases the $\Z_{\geq 0}$-grading comes from the component $\Gamma(G \utimes{B} \mathfrak{u}, \cdots)$.
\end{corollary}
To reconcile with the terminologies in \cite{Bry89}, we remark that $G \utimes{B} \mathfrak{u}$ is nothing but the cotangent bundle of $G/B$.

\begin{proof}
  It suffices to take up the Example \ref{ex:q-analogue} in which $N=\mathfrak{u}$, $P=B$ and $\Psi = \Sigma_B$. Note that
  \begin{itemize}
    \item the collapsing $\pi: G \utimes{B} \mathfrak{u} \to G \cdot \mathfrak{u}$ becomes the Springer resolution of the nilpotent cone in $\mathfrak{g}$, in particular $\pi$ is birational;
    \item the condition on $\mu$ in Theorem \ref{prop:vanishing} reduces to $\mu \in X^*(T)_-$.
  \end{itemize}
  Thus the higher cohomologies in question all vanish and we may apply Corollary \ref{prop:gKF-positivity}.
\end{proof}

\section{Basic functions as Poincaré series}\label{sec:gKF-basic}
\subsection{Geometry of the variety $\mathcal{Y}$}\label{sec:variety-Y}
Let $\Bbbk$, $G$, $B$ and $T$ as be in \S\ref{sec:gKF}. Consider a submonoid $X^*(T)_\ominus$ of $X^*(T)$ such that
$$ X^*(T)_\ominus \subset X^*(T)_- $$
and of the form
$$ X^*(T)_\ominus = \mathcal{C}^* \cap X^*(T) $$
where $\mathcal{C}^*$ is some rational cone in $X^*(T)_\R$ that is strongly convex (i.e. defined over $\Q$ and containing no lines).

To such a $X^*(T)_\ominus$ there is an associated a normal affine toric variety under $T$, denoted by $\Pi$. More precisely, we define the following representation of $T$
$$ R := \bigoplus_{\mu \in X^*(T)_\ominus} \Bbbk_\mu. $$
It is actually a $\Bbbk$-algebra with $T$-action: the multiplication comes from the equality $\Bbbk_{\mu_1} \otimes \Bbbk_{\mu_2} = \Bbbk_{\mu_1 + \mu_2}$ for representations of $T$. One can show that $R$ is a domain finitely generated over $\Bbbk$ (see also the proof of Lemma \ref{prop:Pi-embedding} below). Set
$$ \Pi :=  \Spec(R). $$

We inflate $R$ (resp. $\Pi$) to a $\Bbbk$-algebra (resp. $\Bbbk$-variety) with $B$-action.

\begin{lemma}\label{prop:Pi-embedding}
  There exists an affine $G$-variety $\Xi = \Spec(S)$ together with a $B$-equivariant closed immersion $\Pi \hookrightarrow \Xi$.
\end{lemma}
\begin{proof}
  For any $\lambda \in X^*(T)$, let $V(\lambda)$ be the irreducible representation of $G$ of extremal weight $\lambda$. Consider the diagram
  \begin{gather}\label{eqn:R-S-diag}
  \xymatrix{
    S := \bigoplus_{\mu \in X^*(T)_\ominus} V(-\mu)^\vee \ar@{^{(}->}[r] \ar@{-->>}[d] & \bigoplus_{\lambda \in X^*(T)_+} V(\lambda)^\vee = \Bbbk\left[ \overline{G/U}^\text{aff} \right] \\
    R = \bigoplus_{\mu \in X^*(T)_\ominus} \Bbbk_\mu &
  }\end{gather}
  where $\overline{G/U}^\text{aff}$ denotes the affine closure of the quasi-affine variety $G/U$; see \cite[p.10]{Ti11} for detailed discussions. Here the representations $V(\lambda)^\vee$ are viewed as spaces of regular functions on $\overline{G/U}^\text{aff}$, so it makes sense to talk about multiplication and there is an equality (see \cite[Lemma 2.23]{Ti11}):
  $$ V(\lambda_1)^\vee \cdot V(\lambda_2)^\vee = V(\lambda_1 + \lambda_2)^\vee, \quad \lambda_1, \lambda_2 \in X^*(T)_+ . $$
  This makes $S$ into a $\Bbbk$-algebra. Moreover, $S$ is finitely generated: this follows from the equality above and the fact that $X^*(T)_\ominus$ is finitely generated as a commutative monoid (Gordan's lemma).

  The first row of \eqref{eqn:R-S-diag} respects the $G$-actions. To define the vertical arrow, set
  $$ \mathfrak{a} := \bigoplus_{\mu \in X^*(T)_\ominus} \bigoplus_{\nu >_B \mu} (V(-\mu)^\vee)_\nu $$
  where $\nu >_B \mu$ means the strict inequality in the Bruhat order relative to $B$. One easily checks that $\mathfrak{a}$ is a $B$-stable ideal, and
  $$ S/\mathfrak{a} = \bigoplus_{\mu \in X^*(T)_\ominus} (V(-\mu)^\vee)_\mu = \bigoplus_{\mu \in X^*(T)_\ominus} (V(-\mu)^\vee)^{U^-}, \quad T-\text{equivariantly}. $$

  Also observe that $U$ acts trivially on $S/\mathfrak{a}$ since the $\mathfrak{u}$-action raises weights. It remains to identify $S/\mathfrak{a}$ with $R$. Take $o \in G/U$ corresponding to $1 \cdot U$. For every $\mu$ take the unique $f_\mu \in (V(-\mu)^\vee)_\mu$ satisfying $f_\mu(o) = 1$. Sending $f_\mu$ to $1 \in \Bbbk_\mu$ yields the required $B$-equivariant surjection $S \twoheadrightarrow S/\mathfrak{a} \rightiso R$ as $\Bbbk$-algebras.
\end{proof}

\begin{remark}
  When $X^*(T)_\ominus = \bigoplus_{i=1}^r \Z_{\geq 0} \varpi_i$ for some $\varpi_1, \ldots, \varpi_r$ with $r := \dim T$, there is a more straightforward recipe. Indeed, we have $R = \Sym( \bigoplus_{i=1}^r \Bbbk \varpi_i)$ and
  $$ \Pi = \prod_{i=1}^r \Bbbk_{-\varpi_i} $$
  where we regard each $\Bbbk_{-\varpi_i}$ as a one-dimensional affine $T$-variety. After inflation to $B$ we may regard $\Bbbk_{-\varpi_i}$ as the highest weight subspace of $V(-\varpi_i)$ relative to $B$. Hence $\Pi \hookrightarrow \Xi := \prod_{i=1}^r V(-\varpi_i)$ as $B$-varieties. This is the construction adopted in \cite[2.13]{Bro93}.
\end{remark}

In what follows, we fix a finite-dimensional representation $N$ of $B$, together with an affine $G$-variety $W$ with a closed immersion $N \hookrightarrow W$ of $B$-varieties.
\begin{definition}
  The $B$-equivariant embeddings $\Pi \hookrightarrow \Xi$ (Lemma \ref{prop:Pi-embedding}) and $N \hookrightarrow W$ define the collapsing for homogeneous fiber space
  $$\xymatrix{
    \tilde{\mathcal{Y}} := G \utimes{B} (\Pi \times N) \ar@{^{(}->}[r] \ar[d]_{\pi_0} & G \utimes{B} (\Xi \times W) \simeq G/B \times (\Xi \times W) \ar[d]^{\text{pr}_2} \\
    \mathcal{Y}_0 := G \cdot (\Pi \times N) \ar@{{(}->}[r] & \Xi \times W
  }$$
  All arrows are $G$-equivariant and the rows are closed immersions. Set
  $$ \mathcal{Y} := \text{the normalization of } \mathcal{Y}_0 $$
  so that $\pi_0$ factors through a morphism $\pi: \tilde{\mathcal{Y}} \to \mathcal{Y}$, which is still $G$-equivariant. The space $\tilde{\mathcal{Y}}$ is endowed with a $G \times T \times \Gm$-action, where
  \begin{itemize}
    \item $\Gm$ acts on the component $N$ as dilation by $z^{-1}$, for all $z \in \Gm$;
    \item $G$ acts by left translation as usual and $T$ acts through the toric variety $\Pi$, it is probably better to see this by writing
      $$ \tilde{\mathcal{Y}} = (T \times G) \utimes{T \times B} (\Pi \times N), $$
      the $(T \times B)$-action on $\Pi \times N$ being
      $$ (t, b) \cdot (x,u) = (t \bar{b} x ,bu), \quad (x, u) \in \Pi \times \mathfrak{u}, $$
      where $\bar{b}$ denotes the image of $b \in B$ in $T=B/U$.
  \end{itemize}

  Endow $\mathcal{Y}_0$ with the $G \times T \times \Gm$-action so that $\pi_0$ is equivariant. This induces a natural $G \times T \times \Gm$-action on $\mathcal{Y}$ making $\pi: \tilde{\mathcal{Y}} \to \mathcal{Y}$ equivariant.
\end{definition}

\begin{hypothesis}\label{hyp:N}
  Henceforth it is assumed that
  \begin{enumerate}
    \item the collapsing morphisms $\pi_0: \tilde{\mathcal{Y}} \to \mathcal{Y}_0$ and $\eta: G \utimes{B} N \to G \cdot N$ are both birational;
    \item the condition \eqref{eqn:mu-vanishing-condition} in Theorem \ref{prop:vanishing} (for $N$ and $P=B$) holds for every $\mu \in X^*(T)_\ominus$.
  \end{enumerate}
\end{hypothesis}

\begin{remark}\label{rem:birationality}
  The first condition holds when $N=\mathfrak{u}$ under the adjoint action of $B$, or more generally when $N = N_0 \oplus \mathfrak{u}$ for some $N_0$. Indeed, let $\mathcal{N} \subset \mathfrak{g}$ be the nilpotent cone and $\mathcal{N}_\text{reg}$ be the Zariski open subset of regular nilpotent elements, that is, the elements which belong to a unique Borel subalgebra of $\mathfrak{g}$. Set $\mathfrak{u}_\text{reg} := \mathfrak{u} \cap \mathcal{N}_\text{reg}$. The collapsing restricts to an isomorphism
  $$ G \utimes{B} \mathfrak{u}_\text{reg} \rightiso G \cdot \mathfrak{u}_\text{reg} = \mathcal{N}_\text{reg}. $$

  From this we deduce a cartesian diagram
  $$\xymatrix{
    G \utimes{B} (N_0 \times \mathfrak{u}_\text{reg}) \ar[r]^{\sim} \ar[d] & G \cdot (N_0 \times \mathfrak{u}_\text{reg}) \ar[d] \\
    G \utimes{B} N \ar[r]_{\eta}& G \cdot N \\
  }$$
  in which the vertical arrows are open immersions. The case for $\pi_0$ is similar.
\end{remark}

The following arguments are essentially paraphrases of \cite[2.13]{Bro93}

\begin{lemma}\label{prop:resolution-Y}
  The morphisms $\pi_0$ and $\pi$ are proper and birational. Moreover, $R\pi_* \mathcal{O}_{\tilde{\mathcal{Y}}} = \mathcal{O}_{\mathcal{Y}}$ and $R(\pi_0)_* \mathcal{O}_{\tilde{\mathcal{Y}}}$ is concentrated in degree zero, in the derived categories $D^b(\mathcal{Y})$ and $D^b(\mathcal{Y}_0)$, respectively.
\end{lemma}
\begin{proof}
  By hypothesis $\pi_0$ is birational; it is proper surjective by Lemma \ref{prop:collapsing}. Since the normalization morphism is birational and proper, the same holds for $\pi: \tilde{\mathcal{Y}} \to \mathcal{Y}$.

  Consider the commutative diagram
  $$\xymatrix{
    \tilde{\mathcal{Y}} \ar[r]^{\pi} \ar[d]_p & \mathcal{Y} \ar[d] \\
    G/B \ar[r] & \Spec\;\Bbbk
  }$$
  in which the vertical morphisms are affine. We get
  \begin{gather}\label{eqn:Y-sections-0}
    H^i(\mathcal{Y}, R\pi_* \mathcal{O}_{\tilde{\mathcal{Y}}}) = H^i(G/B, Rp_* \mathcal{O}_{\tilde{\mathcal{Y}}}) = H^i(G/B, p_* \mathcal{O}_{\tilde{\mathcal{Y}}})
  \end{gather}
  for all $i \geq 0$. Upon recalling the description of $\Pi = \Spec R$, the last term can be rewritten via Proposition \ref{prop:bundle-pushforward} as
  \begin{equation}\label{eqn:Y-sections-1}\begin{aligned}
    H^i(G/B, \; p_* p^* \mathcal{O}_{G/B}) & = \bigoplus_{\mu \in X^*(T)_\ominus} H^i(G/B, \; \mathscr{L}_{G/B}(\Sym(N^\vee) \otimes \Bbbk_\mu )) \\
    & = \bigoplus_{\mu \in X^*(T)_\ominus} H^i(G \utimes{B} N, \; r^* \mathscr{L}_{G/B}(\mu))
  \end{aligned}\end{equation}
  where $r$ is the bundle map $G \utimes{B} N \to G/B$. Now we may invoke the Theorem \ref{prop:vanishing} and Hypothesis \ref{hyp:N} to deduce $H^i(\mathcal{Y}, R\pi_* \mathcal{O}_{\tilde{\mathcal{Y}}})=0$ whenever $i \geq 1$. As $\mathcal{Y}$ is affine, this entails that $R\pi_* \mathcal{O}_{\tilde{\mathcal{Y}}}$ is concentrated in degree zero. On the other hand, since $\pi$ is proper and birational, the normality of $\mathcal{Y}$ implies $\pi_* \mathcal{O}_{\tilde{\mathcal{Y}}} = \mathcal{O}_{\mathcal{Y}}$.

  The arguments for $\pi_0$ are similar.
\end{proof}

\begin{proposition}\label{prop:kY}
  There are canonical $G \times T \times \Gm$-equivariant isomorphisms
  \begin{align*}
    \Bbbk[\mathcal{Y}] & \simeq \Gamma(\tilde{\mathcal{Y}}, \mathcal{O}_{\tilde{\mathcal{Y}}}) \\
    & \simeq \bigoplus_{\substack{\mu \in X^*(T)_\ominus \\ k \geq 0}} \Gamma\left(G/B, \; \mathscr{L}_{G/B}(\Bbbk_\mu \otimes \Sym^k (N^\vee)) \right).
  \end{align*}
  On the right hand side, the group $G$ acts on each of the summands $\Gamma(G/B, \cdots)$, whereas the $X^*(T)$-grading by $\mu$ and the $\Z$-grading by $k$ give the $T \times \Gm$-action.
\end{proposition}
\begin{proof}
  By Lemma \ref{prop:resolution-Y}, we have an equivariant identification $\Bbbk[\mathcal{Y}] = H^0(\mathcal{Y}, R\pi_* \mathcal{O}_{\tilde{\mathcal{Y}}})$, which equals $\Gamma(\tilde{\mathcal{Y}}, \mathcal{O}_{\tilde{\mathcal{Y}}})$. From \eqref{eqn:Y-sections-0} and \eqref{eqn:Y-sections-1} with $i=0$ we obtain the required isomorphism. It remains to recall the definition of the $G \times T \times \Gm$-action on $\tilde{\mathcal{Y}}$.
\end{proof}

\begin{proposition}\label{prop:Y-ratsing}
  The $\Bbbk$-varieties $\tilde{\mathcal{Y}}$ and $\mathcal{Y}$ have rational singularities.
\end{proposition}
We refer to \cite[A.1]{Ti11} for the backgrounds about rational singularities.
\begin{proof}
  Consider the case for $\tilde{\mathcal{Y}}$ first. The normal toric variety $\Pi$ is known to have rational singularities \cite[p.76]{Fu93}, hence so do the fibers of $p: \tilde{\mathcal{Y}} \to G/B$. On the other hand, the variety $G/B$ certainly has rational singularities. We conclude by applying \cite[Theorem A.5]{Ti11} to the bundle $p$.

  As for $\mathcal{Y}$, apply \cite[Theorem 1]{Kov00} to $\pi$ together with Lemma \ref{prop:resolution-Y}.
\end{proof}

\subsection{On certain generalized Kostka-Foulkes polynomials}\label{sec:certain-gKF}
Retain the previous notations and choose the data as in Example \ref{ex:basic-fcn}. To be precise, we have now a short exact sequence of connected reductive $\Bbbk$-groups
$$ 1 \to \Gm \to G \to G_0 \to 1 $$
with $G_0$ semisimple, and a finite-dimensional representation $(V,\rho)$ of $G$ such that $\rho(z)=z\cdot\identity$ for all $z \in \Gm$. Define the set with multiplicities
$$ \Psi = \Supp(V) \sqcup \Sigma_B, $$
which is formed by the $T$-weights of the representation $N = (V, \rho|_B) \oplus (\mathfrak{u}, \Ad|_B)$ of $B$. Dualizing $\Gm \hookrightarrow G$ furnishes a surjective homomorphism $\det_G: X^*(T) \to X^*(\Gm)=\Z$.

\begin{definition}\label{def:X-ominus}
  Define the monoid
  \begin{gather}
    X^*(T)_\ominus := \left\{ \mu \in X^*(T)_- : \det_G(\mu) \geq 0 \right\}.
  \end{gather}
  and form the corresponding normal affine toric variety $\Pi := \Spec(R)$. See also the Remark \ref{rem:smaller-monoid}
\end{definition}

\begin{lemma}
  The Hypothesis \ref{hyp:N} is satisfied with the choice of $N$ and $X^*(T)_\ominus$ above.
\end{lemma}
\begin{proof}
  In Remark \ref{rem:birationality} (with $N_0 := V$) we have seen that the first condition in Hypothesis \ref{hyp:N} is satisfied. Let $\theta := \sum_{\alpha \in \Supp(V)} \alpha$, the second condition amounts to
  $$ \forall \mu \in X^*(T)_\ominus, \quad \mu + \theta \in X^*(T)_- . $$
  Since $V$ is a representation of $G$, the sum $\theta$ is fixed by every root reflection, hence $\angles{\theta, \beta^\vee} = 0$ for every $\beta \in \Delta_B$. Therefore $\mu + \theta \in X^*(T)_-$ as required.
\end{proof}

\begin{remark}\label{rem:smaller-monoid}
  In practice one can often work with a smaller monoid. For example, after proving Theorem \ref{prop:basicfun-gen} one will see that the submonoid $\mathcal{C}_\rho \cap X_*(T)_-$ suffices for the study of the coefficients of the basic function $f_\rho$ (cf. Corollary \ref{prop:basic-supp}).
\end{remark}

The results in \S\ref{sec:variety-Y} are thus applicable. Define the morphisms
$$\xymatrix{
  & \tilde{\mathcal{Y}} = G \utimes{B} (\Pi \times N) \ar[ld]_{p} \ar[rd]^{\pi} & \\
  G/B & & \mathcal{Y}
}$$
accordingly. Recall that the $\Gm$-action on $\tilde{\mathcal{Y}}$ comes from the dilation on $N = V \times \mathfrak{u}$. It will be convenient to thicken it to a $\Gm \times \Gm$-action, namely
$$ \forall (z, w) \in \Gm \times \Gm, \quad (z,w) \cdot [g, (\pi, v, u)] = [g, (\pi, z^{-1}v, w^{-1}u)] $$
for all $[g, (\pi, v,u)] \in \tilde{\mathcal{Y}}$, where $(\pi, v, u) \in \Pi \times V \times \mathfrak{u}$. Restriction to the diagonal $\Gm$ gives back the original action.

Introduce indeterminates $X$ and $q$.

\begin{definition}
  Using the $G \times T \times \Gm^2$-action on $\mathcal{Y}$, we define the Poincaré series
  $$ \mathbf{P}(\Bbbk[\mathcal{Y}]; X, q) = \sum_{\lambda \in X^*(T)_-} \sum_{\mu \in X^*(T)_\ominus} m^\mu_{\lambda, \Psi}(X,q) e^\mu [V(\lambda)] $$
  for some $m^\mu_{\lambda, \Psi}(X,q) \in \Z[X,q]$, where $\lambda$ (resp. $\mu$) corresponds to the $G$-action (resp. $T$-action), and $(X,q)$ corresponds to the $\Gm^2$-action described above. Cf. the Definition \ref{def:P-series}.
\end{definition}

\begin{theorem}\label{prop:qq}
  The generalized Kostka-Foulkes polynomials associated to $\Psi$ are recovered via
  $$ m^\mu_{\lambda, \Psi}(q) = m^\mu_{\lambda,\Psi}(q,q), \quad \lambda \in X^*(T)_-, \mu \in X^*(T)_\ominus. $$
\end{theorem}
\begin{proof}
  In view of Hypothesis \ref{hyp:N}, we may apply Corollary \ref{prop:gKF-positivity} to every $\mu \in X^*(T)_\ominus$ to get
  $$ \mathbf{P} \left( \; \bigoplus_{\mu \in X^*(T)_\ominus} \Gamma \left( G \utimes{B} N, \; r^*\mathscr{L}_{G/B}(\mu) \right); q \right) = \sum_{\lambda \in X^*(T)_-} \; \sum_{\mu \in X^*(T)_\ominus} m^\mu_{\lambda, \Psi}(q) e^\mu [V(\lambda)] $$
  where $r: G \utimes{B} N \to G/B$. Furthermore,
  \begin{align*}
    \bigoplus_{\mu \in X^*(T)_\ominus} \Gamma \left( G \utimes{B} N, \; r^*\mathscr{L}_{G/B}(\mu) \right) & \simeq \bigoplus_{\substack{\mu \in X^*(T)_\ominus \\ k \geq 0}} \Gamma\left(G/B, \; \mathscr{L}_{G/B}(\Bbbk_\mu \otimes \Sym^k (N^\vee)) \right) \\
    & \simeq \Bbbk[\mathcal{Y}];
  \end{align*}
  the first isomorphism follows from Proposition \ref{prop:bundle-pushforward} and the second follows from Proposition \ref{prop:kY}. The $G \times T \times \Gm$-actions are respected ($\Gm \hookrightarrow \Gm^2$ diagonally), whence the assertion.
\end{proof}

Next, let $\kappa: G \utimes{B} \mathfrak{u} \to G/B$ and define the $\Bbbk$-algebra
\begin{gather}\label{eqn:Z}
  \mathcal{Z} := \Sym(V^\vee) \otimes \left(\; \bigoplus_{\mu \in X^*(T)_\ominus} \Gamma(G \utimes{B} \mathfrak{u}, \; \kappa^* \mathscr{L}_{G/B}(\mu) ) ) \right).
\end{gather}

The algebra $\mathcal{Z}$ is a direct sum of irreducible finite-dimensional representation of $G$, since both tensor slots are. Moreover, it is equipped with
\begin{compactitem}
  \item the $X^*(T)_\ominus$-grading via $\mu$;
  \item the $\Z_{\geq 0}$-grading on the first tensor slot, by placing $\Sym^k(V^\vee)$ in degree $k$;
  \item the $\Z_{\geq 0}$-grading on the second tensor slot, via the familiar dilation on the fibers of $G \utimes{B} \mathfrak{u} \to G/B$.
\end{compactitem}
They are compatible with the $G$-representation structure. The invariant subalgebra $\mathcal{Z}^G$ is thus graded by $X^*(T)_\ominus \times (\Z_{\geq 0})^2$.

\begin{theorem}\label{prop:Z-P-series}
  There is a canonical isomorphism of $\Bbbk$-algebras $\mathcal{Z} \rightiso \Bbbk[\mathcal{Y}]$ respecting the $G$-actions and gradings. Consequently,
  \begin{align*}
    \mathbf{P}(\mathcal{Z}; X, q) & = \sum_{\lambda \in X^*(T)_-} \sum_{\mu \in X^*(T)_\ominus} m^\mu_{\lambda, \Psi}(X,q) e^\mu [V(\lambda)], \\
    \mathbf{P}(\mathcal{Z}^G; X, q) & = \sum_{\mu \in X^*(T)_\ominus} m^\mu_{0, \Psi}(X, q) e^\mu.
  \end{align*}
\end{theorem}
\begin{proof}
  By \eqref{eqn:proj-formula}, there is an equivariant isomorphism
  $$ \tilde{\mathcal{Y}} = G \utimes{B} (\Pi \times V \times \mathfrak{u}) \simeq V \times \left( G \utimes{B} (\Pi \times \mathfrak{u}) \right). $$
  By Künneth formula, $\Bbbk[\mathcal{Y}] \simeq \Gamma(\tilde{\mathcal{Y}}, \mathcal{O}_{\tilde{\mathcal{Y}}})$ is isomorphic to $\mathcal{Z}$. Indeed, the first tensor slot of $\mathcal{Z}$ is simply $\Bbbk[V]$, whilst the Proposition \ref{prop:kY} and \eqref{eqn:Y-sections-1} in the case $N=\mathfrak{u}$ recognizes the second tensor slot. It is easy to see that the structures of $G$-representations and gradings match, and the equalities of Poincaré series follow.
\end{proof}

\begin{corollary}\label{prop:X-degree}
  For every $\mu \in X^*(T)_\ominus$, we have $m^\mu_{0,\Psi}(X, q) = m^\mu_{0,\Psi}(1, q) X^{\det_G(\mu)}$. Consequently, $m^\mu_{0, \Psi}(q) = m^\mu_{0,\Psi}(1, q) q^{\det_G(\mu)}$.
\end{corollary}
\begin{proof}
  We identify $\mathfrak{u}^\vee$ with the opposite nilpotent radical $\mathfrak{u}^-$ using the Killing form on $\mathfrak{g}_\text{der}$.

  Let $\mu \in X^*(T)$ and $k \in \Z_{\geq 0}$, we contend that
  $$ \Sym^k(V^\vee) \otimes \Gamma \left( G \utimes{B} \mathfrak{u}, \; \kappa^* \mathscr{L}_{G/B}(\mu) ) \right) $$
  contains the trivial representation of $G$ only if $k = \det_G(\mu)$. Indeed, the second tensor slot is
  $$ \Gamma \left( G/B, \mathscr{L}_{G/B}(\Bbbk_\mu \otimes \Sym(\mathfrak{u}^-) ) \right). $$
  One checks that every $z \in \Gm \hookrightarrow G$ acts on this representation as $\mu(z) \cdot \identity = z^ {\det_G(\mu)} \cdot \identity$, since $Z_G$ acts trivially on $\Sym(\mathfrak{u}^-)$. On the other hand, $z$ acts on $\Sym^k(V^\vee)$ as $z^{-k} \cdot \identity$. Hence $k = \det_G(\mu)$ if it contains the trivial $G$-representation.

  In view of Theorem \ref{prop:Z-P-series}, the first assertion follows immediately, and the second one follows from Theorem \ref{prop:qq}.
\end{proof}

\begin{corollary}\label{prop:rat-fun}
  The formal power series $\mathbf{P}(\mathcal{Z}^G; X, q)$ is rational of the form
  \begin{equation}\label{eqn:rat-fun}
    \frac{Q}{\prod_{i=1}^s (1 - q^{d_i} e^{\mu_i} X^{\det_G(\mu_i)})}
  \end{equation}
  for some $s$, where $d_i \geq 0$, $\mu_i \in X^*(T)_\ominus$ for all $i$, and $Q \in \Z[X^*(T), X^{\pm 1}, q^{\pm 1}]$ is a $\Z$-linear combination of monomials of the form $e^\mu X^{\det_G(\mu)} q^d$, for various $\mu$ and $d$.
\end{corollary}
\begin{proof}
  Note that $\mathcal{Z}^G \simeq \Bbbk[\mathcal{Y}]^G$ is finitely generated over $\Bbbk$ \cite[p.217]{Ti11}. The assertion then follows from the multi-graded Hilbert-Serre theorem \cite[Theorem 6.37]{BG09}; the term $X^{\det_G(\mu)}$ comes from Corollary \ref{prop:X-degree}.
\end{proof}

\begin{corollary}
  The $\Bbbk$-variety $\Spec(\mathcal{Z}^G)$ has rational singularities. In particular, $\mathcal{Z}^G$ is Cohen-Macaulay.
\end{corollary}
\begin{proof}
  By Proposition \ref{prop:Y-ratsing}, the affine variety $\mathcal{Y}$ has rational singularities. From \cite{Bou87} we conclude that the same holds for $\Spec(\mathcal{Z}^G) \simeq \mathcal{Y}/\!/G$ (the categorical quotient \cite[II. \S 4.3]{AGIV}).
\end{proof}

Denote by $\omega$ the graded dualizing module of $\mathcal{Z}^G$. By a celebrated theorem of Stanley \cite[Theorem 6.41]{BG09}, there is a simple equation relating the rational functions $\mathbf{P}(\mathcal{Z}^G; X, q)$ and $\mathbf{P}(\omega; X, q)$. A detailed analysis of $\omega$ will be needed in order to elucidate this functional equation.

\subsection{Coefficients of the basic function}\label{sec:coeff-basicfun}
In this subsection we revert to the setup in \S\ref{sec:basicfun}: $F$ is a nonarchimedean local field and we have a short exact sequence of split connected reductive $F$-groups
$$ 1 \to G_0 \to G \xrightarrow{\det_G} \Gm \to 1 $$
in which $G_0$ is semisimple. Fix a Borel pair $(B,T)$ for $G$ defined over $F$, as well as a hyperspecial subgroup $K \subset G(F)$.

On the dual side, these data give rise to the groups over $\Bbbk = \C$
$$ 1 \to \Gm \to \hat{G} \to \widehat{G}_0 \to 1, $$
and $\hat{G}$ is endowed with the dual Borel pair $(\hat{B}, \hat{T})$. We have a finite-dimensional representation $(\rho, V)$ of $\hat{G}$ such that $\rho(z) = z\cdot\identity$ for all $z \in \Gm$. There is also the surjective homomorphism
$$ \det_G: X_*(T) = X^*(\hat{T}) \longrightarrow X^*(\Gm) = X_*(\Gm) = \Z. $$

The formalism in \S\ref{sec:certain-gKF} applies to $\hat{G}$, etc. We define the corresponding objects $N = V \oplus \mathfrak{u}$, $\Psi := \Supp(N)$ and $m^\mu_{\lambda, \Psi}(q)$ (Definition \ref{def:gKF}) for $\lambda \in X_*(T)_-$, $\mu \in X_*(T)$, etc. Also, $X_*(T)_\ominus$ is the submonoid of $X_*(T)_-$ defined by $\det_G \geq 0$.

On the other hand, in Definition \ref{def:basic} we defined the basic function $f_{\rho, X}$ with the indeterminate $X$. Crucial in that definition are the coefficients $c_\mu(q_F)$ indexed by $\mu \in X_*(T)_-$, where $c_\mu(q) \in \Z[q^{-1}]$. According to \eqref{eqn:c_mu}, $c_\mu \neq 0$ only if $\mu \in X_*(T)_\ominus$.

\begin{theorem}\label{prop:basicfun-gen}
  For all $\mu \in X_*(T)_\ominus$ we have
  $$ c_\mu(q^{-1}) = m^\mu_{0, \Psi}(q) q^{-\det_G(\mu)}. $$
  Equivalently, the equalities below hold in $\Z[X_*(T)_\ominus][q,X]$:
  \begin{align*}
    \sum_{\mu \in X_*(T)_\ominus} c_\mu(q^{-1}) e^\mu X^{\det_G(\mu)} & = \sum_{\mu \in X_*(T)_\ominus} m^\mu_{0, \Psi}(X,q) e^\mu \\
    & = \sum_{\mu \in X_*(T)_\ominus} m^\mu_{0, \Psi}(q) q^{-\det_G(\mu)} X^{\det_G(\mu)} e^\mu.
  \end{align*}
  In fact, they equal the Poincaré series $\mathbf{P}(\mathcal{Z}^{\hat{G}}; X,q)$ where $\mathcal{Z}$ is defined in \eqref{eqn:Z}.
\end{theorem}

Substituting
\begin{align*}
  q & \longmapsto q_F^{-1}, \\
  e^\mu & \longmapsto q_F^{-\angles{\rho_{B^-}, \mu}} \mathbbm{1}_{K\mu(\varpi)K}
\end{align*}
in these Poincaré series yields the basic function $f_{\rho, X}: K \backslash G(F) / K \to \Z[q_F^{\pm 1/2}, X]$.

\begin{proof}
  We use \eqref{eqn:c_mu} to write $\sum_\mu c_\mu(q^{-1}) e^\mu X^{\det_G(\mu)}$ as
  \begin{gather*}
    \sum_{k \geq 0} \; \sum_{\substack{\lambda \in X_*(T)_- \\ \det_G(\lambda)=k}} \; \sum_{\substack{\mu \in X_*(T)_\ominus \\ \mu \leq \lambda}} K_{\lambda, \mu}(q) \text{mult}(\Sym^k(\rho) : V(\lambda)) e^\mu X^k.
  \end{gather*}

  The $\leq$ above is the Bruhat order relative to $B^-$. The following manipulations are justified since they are ``finitary'' for a fixed $k$. Apply Corollary \ref{prop:K-as-P-series} to turn the sum into
  \begin{gather*}
    \sum_{k \geq 0} \; \sum_{\substack{\mu \leq \lambda \\ \det_G(\mu) = k}}  \mathbf{P}( \mathcal{Z}_{\lambda, \mu} ; q) \text{mult}(\Sym^k(\rho) : V(\lambda)) e^\mu X^k
  \end{gather*}
  where we define the homogeneous fibration $\kappa: \hat{G} \utimes{\hat{B}} \hat{\mathfrak{u}} \to \hat{G}/\hat{B}$ and the $\Z_{\geq 0}$-graded $\Bbbk$-vector space
  $$ \mathcal{Z}_{\lambda, \mu} := \Hom_{\hat{G}} \left( V(\lambda), \Gamma\left( \hat{G} \utimes{\hat{B}} \mathfrak{\hat{u}}, \; \kappa^* \mathscr{L}_{\hat{G}/\hat{B}}(\mu) \right) \right). $$

  Observe that the condition $\det_G(\mu)=k$ can be removed, since $\mu \leq \lambda$ implies $\det_G(\lambda) = \det_G(\mu)$, and $\text{mult}(\Sym^k(\rho) : V(\lambda)) \neq 0$ implies $\det_G(\lambda)=k$. It has been observed in Example \ref{ex:q-analogue} that $\mu \leq \lambda$ can also be removed, since otherwise we get $\mathbf{P}( \mathcal{Z}_{\lambda, \mu} ; q) = K_{\lambda, \mu}(q) = 0$. The resulting expression is
  $$ \sum_{k \geq 0} \; \sum_{\lambda \in X_*(T)_-} \mathbf{P} \left( \bigoplus_{\mu \in X_*(T)_\ominus} \mathcal{Z}_{\lambda, \mu} ; q \right) \text{mult}(\Sym^k(\rho) : V(\lambda)) X^k , $$
  where $\bigoplus_\mu \mathcal{Z}_{\lambda, \mu}$ is naturally $X_*(T)_\ominus \times \Z_{\geq 0}$-graded. The inner sum over $\lambda$ yields the Poincaré series of
  $$ \Hom_{\hat{G}} \left( \Sym^k(\rho), \Gamma\left( \hat{G} \utimes{\hat{B}} \mathfrak{\hat{u}}, \; \kappa^* \mathscr{L}_{\hat{G}/\hat{B}}(\mu) \right) \right), $$
  equivalently,
  $$ \left( \Sym^k(V^\vee) \otimes \Gamma\left( \hat{G} \utimes{\hat{B}} \mathfrak{\hat{u}}, \; \kappa^* \mathscr{L}_{\hat{G}/\hat{B}}(\mu) \right) \right)^{\hat{G}}. $$

  Taking the sum over $k \geq 0$, we arrive at the Poincaré series $\mathbf{P}(\mathcal{Z}^{\hat{G}}; X, q)$. On the other hand, Theorem \ref{prop:Z-P-series} asserts that $\mathbf{P}(\mathcal{Z}^{\hat{G}}; X, q) =  \sum_{\mu \in X_*(T)_\ominus} m^\mu_{0, \Psi}(X,q) e^\mu$.

  The last equality follows by Corollary \ref{prop:X-degree}.
\end{proof}

\begin{corollary}\label{prop:c_mu-rat-fun}
  The expression $\sum_{\mu \in X_*(T)_\ominus} c_\mu(q^{-1}) e^\mu X^{\det_G(\mu)}$ in Theorem \ref{prop:basicfun-gen} is rational of the form \eqref{eqn:rat-fun}, with $X_*(T)_{\cdots}$ in place of $X^*(T)_{\cdots}$.
\end{corollary}
\begin{proof}
  Immediate from Corollary \ref{prop:rat-fun}.
\end{proof}

\begin{remark}
  The indeterminate $X$ is somehow redundant. The result above can be written as
  $$ \sum_{\mu \in X_*(T)_\ominus} c_\mu(q^{-1}) e^\mu = \sum_{\mu \in X_*(T)_\ominus} m^\mu_{0, \Psi}(q) q^{-\det_G(\mu)} e^\mu $$
 in $\Z \llbracket X_*(T)_\ominus \rrbracket [q]$.
\end{remark}

\section{Examples}\label{sec:examples}
\subsection{The standard $L$-factor of Tamagawa-Godement-Jacquet}\label{sec:GJ}
Fix $n \in \Z_{\geq 1}$. In Ngô's recipe \S\ref{sec:Ngo}, take
\begin{itemize}
  \item $G_0 := \SL(n)$, viewed as a subgroup of $G := \GL(n)$;
  \item $(B, T)$: the standard Borel pair of $\GL(n)$, namely $B$ (resp. $T$) is the subgroup of upper triangular (resp. diagonal) matrices;
  \item $(B_0, T_0) := (B \cap G_0, T \cap G_0)$, which is a Borel pair of $G_0$;
  \item $\varepsilon_1, \ldots, \varepsilon_n$: the standard basis of $X_*(T)$, namely $\varepsilon_i: \Gm \to T = \Gm^n$ is $\identity$ at the $i$-th slot of $\Gm^n$, and trivial elsewhere;
  \item $\bar{\xi} \in X_*(T_{0,\text{ad}})$ is the cocharacter obtained by composing $\Gm \xrightarrow{\varepsilon_n} T \twoheadrightarrow T/Z_G = T_{0,\text{ad}}$.
\end{itemize}

Choose the usual $\mathfrak{o}_F$-model of $G$. The hyperspecial subgroup $K$ of $G(F)$ is $\GL(n,\mathfrak{o}_F)$. We have $\widehat{G_0} = \PGL(n,\C)$, $\widehat{G_0}_{,SC} = \SL(n,\C)$ and $\hat{G} = \GL(n,\C)$. To $\bar{\xi}$ is associated the standard representation $\rho_{\bar{\xi}}: \SL(n,\C) \hookrightarrow \GL(n,\C)$ for $\widehat{G_0}_{,SC}$, which yields the tautological projective representation $\bar{\rho}: \PGL(n,\C) \xrightarrow{\identity} \PGL(n,\C)$ for $\widehat{G_0}$. Now gaze at the diagram \eqref{eqn:lifting}: the lifted representation $\rho$ is nothing but the standard representation on $\C^n$
$$ \text{Std}: \hat{G} = \GL(n,\C) \xrightarrow{\identity} \GL(n,\C), $$
for tautological reasons. Moreover, its highest weight relative to $B^-$ is the cocharacter
$$ \xi = \varepsilon_n: \Gm \to T . $$

In this case, $\det_G$ is simply the determinant $\det: \GL(n) \to \Gm$. It induces the homomorphism $\det: X_*(T) \to X_*(\Gm)=\Z$ sending $\sum_i a_i \varepsilon_i$ to $\sum_i a_i$.

Thus we recover the setting of Tamagawa \cite{Ta63} and Godement-Jacquet \cite{GJ72}. Their calculations in the unramified setting can actually be deduced from our formalism in \S\ref{sec:L-basic}, as explained below.

The weights of $\text{Std}$ are $\varepsilon_1, \ldots, \varepsilon_n$, each with multiplicity one. We take the corresponding weight vectors $v_1, \ldots, v_n$ to be the standard basis of $\C^n$. Let $X_*(T)_\ominus = X^*(\hat{T})_\ominus$ be as in Definition \ref{def:X-ominus}.

The following criterion is suggested by Casselman.
\begin{proposition}\label{prop:Sym-irred}
  Suppose temporarily that $\hat{G}$ is any connected reductive $\C$-group and $\hat{B}$ is a Borel subgroup with unipotent radical $\hat{U}$. Let $(\rho, V)$ be an irreducible representation of $\hat{G}$ of highest weight $\varepsilon$ relative to $\hat{B}^-$. Let $v_- \neq 0$ be a highest vector in $V$ of weight $\varepsilon$ relative to $\hat{B}^-$. Then $\Sym^k(V)$ is irreducible for all $k \geq 0$ (with highest weight $k\varepsilon$) if and only if $\hat{U} \cdot (\C v_-)$ is dense in $V$.
\end{proposition}
\begin{proof}
  Write $\mathbb{P}(V)$ for the projective space and $[v_-] \in \mathbb{P}(V)$ for the line containing $v_-$. Notice that $\hat{G}  \cdot [v_-]$ is always closed. The density of $\hat{U} \cdot (\C v_-)$ is equivalent to
  $$ \hat{G}  \cdot [v_-] = \mathbb{P}(V). $$

  Let $\check{v}_+$ be the highest vector of $V^\vee$ relative to $\hat{B}$ such that $\angles{\check{v}_+, v_-}=1$, of weight $-\varepsilon$. It is known that $\Stab_{\hat{G}} [v_-]$ and $\Stab_{\hat{G}} [\check{v}_+]$ are opposite parabolic subgroups (see \cite[1.1]{BLV86}). Hence the density condition is equivalent to $\hat{G}  \cdot [\check{v}_+] = \mathbb{P}(V^\vee)$, or to the density of $\hat{U}^- \cdot (\C \check{v}_+)$ in $V^\vee$.

  To decompose $\Sym(V)$ it suffices to describe $\Sym(V)^{\hat{U}^-} = \C[V^\vee]^{\hat{U}^-}$.

  For the ``if'' part, we shall prove that $\C[V^\vee]^{\hat{U}^-} = \C[v_-]$ (the polynomial algebra generated by $v_-$). The inclusion $\supset$ is evident. As for $\subset$, the discussion above implies that every $\hat{U}^-$-invariant regular function $f$ on $V^\vee$ is a polynomial in $v_-$,  thereby establishing our claim for all $\Sym^k(V)$.

  For the ``only if'' part, let $Z$ be the Zariski closure of $\hat{U}^- \cdot (\C \check{v}_+)$ in $V^\vee$: it is defined by a $\hat{G}$-stable ideal $I$.  We have
  $$ \C[V^\vee]^{\hat{U}^-}/I^{\hat{U}^-} = \C[Z]^{\hat{U}^-} = \C[v_-] = \C[V^\vee]^{\hat{U}^-} $$
  where the first equality stems from \cite[Lemma D.1]{Ti11} and the second is the restriction to $\C \check{v}_+$. Therefore $I^{\hat{U}^-}=\{0\}$, so $I=\{0\}$ by highest weight theory. Hence $Z=V^\vee$.
\end{proof}

\begin{lemma}\label{prop:c_mu-Std}
  For $\mu = \sum_{i=1}^n a_i \varepsilon_i \in X_*(T)_\ominus$ with $k := \det(\mu)$, the coefficient $c_\mu(q)$ defined in \eqref{eqn:c_mu} satisfies
  $$ c_\mu(q^{-1})
    = \begin{cases}
    q^{k \cdot \frac{n-1}{2} - \angles{\rho_{B^-}, \mu}}, & \text{if } a_1, \ldots, a_n \geq 0, \\
    0, & \text{otherwise}.
  \end{cases} $$
\end{lemma}
\begin{proof}
  We begin by showing that $\Sym^k(\text{Std})$ is irreducible of highest weight $\lambda := k\varepsilon_n$. There are several ways to do this, eg. by Schur-Weyl duality; here we opt for the invariant-theoretic approach via Proposition \ref{prop:Sym-irred}. Indeed, the orbit $\hat{U} \cdot (\C v_n)$ is open dense since the $n$-th column of elements in $\hat{U}$ is arbitrary except for the last $1$.

  For $\lambda$ as above, we always have $\mu \leq \lambda$. Therefore \eqref{eqn:c_mu} reduces to
  $$ c_\mu(q^{-1}) = K_{\lambda,\mu}(q). $$
  Specialized at $q=1$, we obtain the multiplicity of the weight $\mu$ in $\Sym^k(\text{Std})$ (recall Example \ref{ex:q-analogue}), that is, the cardinality of
  $$ \left\{ (b_1, \ldots, b_n) \in \Z_{\geq 0}^n : \sum_{i=1}^n b_i \varepsilon_i = \mu \right\}, $$
  which is $1$ if $a_1, \ldots, a_n \geq 0$, and zero otherwise.

  Since $c_\mu(q^{-1}) \in \Z_{\geq 0}[q]$, it must be a monomial in $q$ when $a_1, \ldots, a_n \geq 0$, and zero otherwise. In the former case, we have seen that
  $$ \deg K_{\lambda,\mu}(q) = \angles{\lambda-\mu, \rho_{B^-}} = \frac{n-1}{2} \cdot k - \angles{\rho_{B^-}, \mu}. $$
  This concludes the proof.
\end{proof}

\begin{theorem}
  Let
  \begin{align*}
    \Lambda & := \left\{ \sum_{i=1}^n a_i \varepsilon_i \in X_*(T)_- : a_1, \ldots, a_n \geq 0  \right\} \\
    & = \left\{ \sum_{i=1}^n a_i \varepsilon_i \in X_*(T) : 0 \leq a_1 \leq \cdots \leq a_n \right\}.
  \end{align*}
  The basic function in the Tamagawa-Godement-Jacquet case is given by
  $$ f_{\mathrm{Std},X} = \sum_{\mu \in \Lambda} q_F^{-\frac{n-1}{2} \det\mu} \cdot  \mathbbm{1}_{K\mu(\varpi)K} X^{\det\mu}. $$
\end{theorem}
\begin{proof}
  Immediate from Definition \ref{def:basic} and Lemma \ref{prop:c_mu-Std}.
\end{proof}

Note that $\Lambda$ equals the monoid $\mathcal{C}_\rho \cap X^*(T)_-$ appearing in Corollary \ref{prop:basic-supp}, for $\rho=\text{Std}$.

Set $\lambda_i := \varepsilon_{n-i+1} + \cdots + \varepsilon_n$ for $1 \leq i \leq n$. Then
\begin{gather}\label{eqn:C-basis-GJ}
  \Lambda = \Z_{\geq 0} \lambda_1 \oplus \cdots \oplus \Z_{\geq 0} \lambda_n.
\end{gather}

\begin{theorem}
  The generating function in Corollary \ref{prop:c_mu-rat-fun} takes the form
  $$ \sum_{\mu \in X_*(T)_\ominus} c_\mu(q^{-1}) e^\mu X^{\det(\mu)} = \prod_{i=1}^n \left( 1 - q^{\frac{i(i-1)}{2}} e^{\lambda_i} X^i \right)^{-1}. $$
\end{theorem}
\begin{proof}
  In Lemma \ref{prop:c_mu-Std}, the exponent of $q$ in $c_\mu(q^{-1})$ is linear in $\mu \in \Lambda$. Since $\det(\lambda_i)=i$ and
  $$ \angles{\lambda_i, \rho_{B^-}} = \underbrace{\frac{n-1}{2} + \cdots + \frac{n-2i+1}{2}}_{i \text{ terms}} = \dfrac{i(n-i)}{2}, $$
  the coefficient of $q$ in $c_{\lambda_i}(q^{-1})$ is $\frac{i(i-1)}{2}$. We conclude by using \eqref{eqn:C-basis-GJ}.
\end{proof}

The specialization $s = -\frac{n-1}{2}$, i.e. $X \leadsto q_F^{-s} = q_F^{\frac{n-1}{2}}$ yields
$$ f_{\text{Std}, -(n-1)/2} = \sum_{\mu \in \Lambda} \mathbbm{1}_{K\mu(\varpi)K} \quad \in \mathcal{H}_\text{ac}(G(F)\sslash K). $$

Recall the elementary fact
$$ \text{Mat}_{n \times n}(\mathfrak{o}_F) \cap \GL(n,F) = \bigsqcup_{\mu \in \Lambda} K\mu(\varpi)K. $$
Hence $f_{\text{Std}, -(n-1)/2} = \mathbbm{1}_{\text{Mat}_{n \times n}(\mathfrak{o}_F)}|_{\GL(n,F)}$, which is a very familiar element of the Schwartz-Bruhat space of $\text{Mat}_{n \times n}(F)$.

Let $(\pi, V_\pi)$ be a $K$-unramified irreducible representation of $G(F)$. Then the Proposition \ref{prop:basic-L} implies
$$ L\left( -\frac{n-1}{2}, \pi, \text{Std} \right) = \Tr\left( \pi(\mathbbm{1}_{\text{Mat}_{n \times n}(\mathfrak{o}_F)}) \big| V_\pi \right) $$
upon twisting $\pi$ by $|\det|_F^s$ for $\Re(s) \gg 0$ to ensure convergence. By the Remark \ref{rem:basic-L}, it is also equal to $\int_{\GL(n,F)} a(x) \mathbbm{1}_{\text{Mat}_{n \times n}(\mathfrak{o}_F)}(x) \dd x$ where $a(x)$ is the matrix coefficient $\angles{\check{v}, \pi(x)v}$ for $\pi$ with $K$-fixed $v,\check{v}$ and $\angles{\check{v},v}=1$.

All these facts are already contained in \cite{Ta63,GJ72}, although they were derived in a quite different manner there.

\subsection{The spinor $L$-factor for $\GSp(4)$}\label{sec:GSp4}
Fix a base field $\Bbbk$ of characteristic $\neq 2$. We begin by reviewing the structure of the symplectic similitude group $\GSp(4)$ over $\Bbbk$.

Consider a symplectic $\Bbbk$-vector space of dimension $4$, equipped with the symplectic form $\angles{\cdot|\cdot}$ and a basis $e_{-2}, e_{-1}, e_1, e_2$ with
\begin{compactitem}
  \item $\angles{e_i | e_{-j}} = \delta_{i,j}$ for $1 \leq i,j \leq 2$;
  \item $\angles{e_1 | e_2} = \angles{e_{-1} | e_{-2}} = 0$.
\end{compactitem}

Let $g \mapsto {}^* g$ be the transpose anti-automorphism such that $\angles{gx|y} = \angles{x|{}^* g y}$. Using this basis, we have the identifications
\begin{align*}
  G_0 := \Sp(4) &= \{g \in \GL(4) : {}^t g J g = J\}, \\
  G := \GSp(4) &= \{ g \in \GL(4) : \exists \sigma(g) \in \Gm, \; {}^t g J g = \sigma(g) J \} \\
  & = \{(g,\sigma(g)) \in \GL(4) \times \Gm : {}^* g g = \sigma(g) \},
\end{align*}
where
$$ J := \begin{pmatrix}
  & & & -1 \\
  & & -1 & \\
  & 1 & & \\
  1 & & &
\end{pmatrix} \quad \text{ and } {}^* g = J^{-1} \cdot {}^t g \cdot J. $$

The second description of $\GSp(4)$ makes it a closed $\Bbbk$-algebraic subgroup of $\GL(4)$. We have the \textit{similitude character} $\sigma: \GSp(4) \to \Gm$ which maps $g$ to the element $\sigma(g)$. The center of $G$ coincides with the center $\Gm$ of $\GL(4)$.

The standard Borel pair for $\GL(4)$ induces a Borel pair $(B,T)$ for $G$ (resp. $(B_0,T_0)$ for $G_0$) by taking intersections. In particular, $B$ is upper triangular and $T$ is diagonal.

When $\Bbbk = F$ is a non-archimedean local field of characteristic $\neq 2$, we take the vertex in the Bruhat-Tits building of $G$ arising from the self-dual lattice $\bigoplus_{\pm i = 1,2} \mathfrak{o}_F e_i$ and define the hyperspecial subgroup $K$ of $G(F)$ accordingly.

Choose the standard basis $\check{\varepsilon}_1, \ldots, \check{\varepsilon}_4$ for $X_*(\Gm^4)$ where $\Gm^4 \hookrightarrow \GL(4)$. Let $\varepsilon_1, \ldots, \varepsilon_4$ be its dual basis of $X^*(\Gm^4)$. Then $X_*(T)$ is the subgroup of $X_*(\Gm^4)$ defined by the equation $\varepsilon_1 - \varepsilon_2 - \varepsilon_3 + \varepsilon_4 = 0$. In what follows, we identity each $\varepsilon_i$ with its image in $X^*(T)$.

\begin{compactitem}
  \item The elements of $\Sigma_B$ are: $\alpha := \varepsilon_1 - \varepsilon_2$, $\beta := \varepsilon_2 - \varepsilon_3$, $\alpha+\beta = \varepsilon_1 - \varepsilon_3$, $2\alpha+\beta = 2\varepsilon_1 - \varepsilon_2 - \varepsilon_3$.
  \item The elements of $\Delta_B^\vee$ are: $\check{\alpha} = \check{\varepsilon}_1 - \check{\varepsilon}_2 + \check{\varepsilon}_3 - \check{\varepsilon}_4$, $\check{\beta} = \check{\varepsilon}_2 - \check{\varepsilon}_3$. They generate $X_*(T_0)$.
  \item The similitude character $\sigma$ restricted to $T$ is $\mu := \varepsilon_1 + \varepsilon_4 = \varepsilon_2 + \varepsilon_3$.
  \item The inclusion of the center $\Gm \hookrightarrow G$ is $\check{\mu} := \check{\varepsilon}_1 + \check{\varepsilon}_2 + \check{\varepsilon}_3 + \check{\varepsilon}_4 \in X_*(T)$.
  \item Define $\check{\gamma} := \check{\varepsilon}_1 + \check{\varepsilon}_2 \in X_*(T)$, then we have $\mu \circ \check{\gamma} = \identity$. This implies
    $$ X_*(T) = \Z\check{\alpha} \oplus \Z\check{\beta} \oplus \Z\check{\gamma}. $$
    For example, we have $\check{\mu} = -\check{\alpha} - 2\check{\beta} + 2\check{\gamma}$.
\end{compactitem}

The standard representation $(\rho, V)$ is simply the inclusion $\rho: G \hookrightarrow \GL(4)$, which satisfies $\rho \circ \mu = \identity$. Its weights are
\begin{align*}
  \xi & = \varepsilon_1 = \alpha + \frac{\beta + \mu}{2}, \\
  \xi-\alpha & = \varepsilon_2, \\
  \xi-\alpha-\beta & = \varepsilon_3, \\
  \xi-2\alpha-\beta & = \varepsilon_4,
\end{align*}
where $\varepsilon_1$ is highest relative to $B$ and $\varepsilon_4$ is highest relative to $B^-$; in particular $\rho$ is irreducible. Each weight has multiplicity one.

In general, the dual group of $\GSp(2n)$ is the split $\GSpin(2n+1)$, which carries the irreducible spin representation $(\text{spin}, V)$ with $\dim V = 2^n$; here it is convenient to take the dual group over $\Bbbk$. A special feature of $\GSp(4)$ is that it is isomorphic to its dual $\GSpin(5)$, under which the representation $\text{spin}$ becomes the standard representation $\GSp(4) \hookrightarrow \GL(4)$. This self-duality amounts to the existence of an isomorphism $\Phi: X^*(T) \rightiso X_*(T)$ such that $\Phi$ sends simple roots to simple coroots, and so does its transpose ${}^t \Phi$. By considerations of root lengths, we must have $\alpha \mapsto \check{\beta}$ and $\beta \mapsto \check{\alpha}$. On the other hand, $\Ker\alpha \cap \Ker\beta = \Z\check{\mu}$ and $(\Z\check{\alpha} \oplus \Z\check{\beta})^\perp = \Z\mu$, thus $\mu \mapsto \pm\check{\mu}$. There is an involution of $\GSp(4)$ that is identity on $\Sp(4)$ and flips the similitude character: simply take $g \mapsto {}^* g^{-1}$. Hence we may assume $\mu \mapsto \check{\mu}$. This completely determines $\Phi: X^*(T) \rightiso X_*(T)$ and one sees that
$$ \Phi(\varepsilon_1) = \Phi(\alpha) + \frac{\Phi(\beta)+\Phi(\mu)}{2} = \check{\beta} + \frac{\check{\alpha}+\check{\mu}}{2} = \check{\gamma}. $$

Now let us take $\Bbbk = F$ be a non-archimedean local field of characteristic $\neq 2$ and define $G$, $G_0$, etc. The dual groups $\widehat{G_0}$, $\hat{G}$, the spin or standard representation $\text{spin}: \hat{G} \to \GL(4,\C)$ and the based root datum for $\hat{G}$ are described as above with $\Bbbk = \C$.

In Ngô's recipe \S\ref{sec:Ngo}, we start from $\bar{\xi}: \Gm \xrightarrow{\xi} T \twoheadrightarrow T_{0,\text{ad}}$ where $\xi \in X_*(T) = X^*(\hat{T})$ is the highest weight of the standard representation of $\hat{G}$, relative to $B^-$. As in \S\ref{sec:GJ}, the resulting framework is $G = \GSp(4)$ with $\det_G$ being the similitude character $\sigma: G \to \Gm$. The corresponding irreducible representation $\rho$ of $\hat{G}$ is $\text{spin}: \GSp(4,\C) \to \GL(4,\C)$. We have to describe its weights first.

Recall that $\hat{G}$ is isomorphic to $\GSp(4,\C)$, therefore $X^*(\hat{T})$ is isomorphic to $X^*(T)$. The highest weight $\xi \in X^*(\hat{T})$ of $\text{spin}$ relative to $\hat{B}^-$ is mapped to $\varepsilon_4 \in X^*(T)$, the highest weight of the standard representation of $\GSp(4)$ relative to $B^-$. The identification is realized as follows.

$$\xymatrix@C=5pc{
  X^*(\hat{T}) \ar@{=}[r]_{\text{canonical}} \ar@/^1.5pc/[rr] & X_*(T) \ar[r]^{\sim}_{\Phi^{-1}} & X^*(T)
}$$
where $\Phi$ is the self-duality isomorphism described before. Under $\Phi$ we have
\begin{align*}
  \varepsilon_1 & \longmapsto \check{\gamma} & = \check{\varepsilon}_1 + \check{\varepsilon}_2, \\
  \varepsilon_2 & \longmapsto -\check{\beta} + \check{\gamma} & = \check{\varepsilon}_1 + \check{\varepsilon}_3, \\
  \varepsilon_3 & \longmapsto -\check{\alpha} - \check{\beta} + \check{\gamma} & = \check{\varepsilon}_2 + \check{\varepsilon}_4, \\
  \varepsilon_4 & \longmapsto -\check{\alpha} - 2\check{\beta} + \check{\gamma} & = \check{\varepsilon}_3 + \check{\varepsilon}_4.
\end{align*}

\begin{proposition}\label{prop:C-spin}
  Define the cone $\mathcal{C}_\mathrm{spin} \subset X_*(T)_\R$ as in Corollary \ref{prop:basic-supp}. Then
  \begin{align*}
    \mathcal{C}_\mathrm{spin} & = \left\{ \sum_{i=1}^4 x_i \check{\varepsilon}_i : x_1 + x_4 = x_2 + x_3, \; x_1, x_2, x_3 \geq 0, x_2+x_3 \geq x_1 \right\}, \\
    \mathcal{C}_\mathrm{spin} \cap X_*(T)_- & = \left\{ \sum_{i=1}^4 x_i \check{\varepsilon}_i : x_1 + x_4 = x_2 + x_3, \; 0 \leq x_1 \leq x_2 \leq x_3 \right\} \\
    & = \Z_{\geq 0} (\check{\varepsilon}_3 + \check{\varepsilon}_4) \oplus \Z_{\geq 0} (\check{\varepsilon}_2 + \check{\varepsilon}_3 + 2\check{\varepsilon}_4) \oplus \Z_{\geq 0} (\check{\varepsilon}_1 + \check{\varepsilon}_2 + \check{\varepsilon}_3 + \check{\varepsilon}_4).
  \end{align*}
\end{proposition}
\begin{proof}
  Recall that $\varepsilon_1 + \varepsilon_4 = \varepsilon_2 + \varepsilon_4$ defines $X_*(T)$. We have $\sum_{i=1}^4 x_i \check{\varepsilon}_i \in \mathcal{C}_\text{spin}$ if and only if the equation
  $$
    \begin{pmatrix}
      1 & 1 & 0 & 0 \\
      1 & 0 & 1 & 0 \\
      0 & 1 & 0 & 1 \\
      0 & 0 & 1 & 1 \\
    \end{pmatrix}
    \begin{pmatrix}
      a \\ b \\ c \\ d
    \end{pmatrix} =
    \begin{pmatrix}
      x_1 \\ x_2 \\ x_3 \\ x_4
    \end{pmatrix},
    \qquad a, b, c, d \geq 0
  $$
  is solvable. The solutions are given by $c = -x_1 + x_2 + x_3 - d$, $b = x_3 - d$, $a = x_1 - x_3 + d$. Non-negative solutions exist if and only if there exists $d \geq 0$ such that
  \begin{gather*}
    -x_1 + x_2 + x_3 \geq d \geq -x_1 + x_3, \\
    x_3 \geq d \geq -x_1 + x_3.
  \end{gather*}
  From those inequalities we derive $x_1, x_2 \geq 0$; the non-negativity of $d$ then implies $x_2 + x_3 \geq x_1$ and $x_3 \geq 0$, in which case the non-negative solutions are parametrized by
  $$ \min\{ -x_1 + x_2 + x_3,\; x_3 \} \geq d \geq \max\{-x_1 + x_3, 0 \}. $$

  Taking intersection with $X_*(T)_-$ imposes the constraint $0 \leq x_1 \leq x_2 \leq x_3$, which implies $x_2 + x_3 \geq x_1$. Note that $x_4 = -x_1 + x_2 + x_3 \geq x_3$. This gives the first description of $\mathcal{C}_\mathrm{spin} \cap X_*(T)_-$; the second follows in a routine manner.
\end{proof}

Unlike the case of standard $L$-factors, we are unable to determine the coefficients $c_\mu(q)$ completely in this article. In what follows, we discuss some relations to Satake's work \cite{Sat63} on the spinor $L$-factor for $\GSp(4)$.

\begin{enumerate}
  \item Firstly, set
    $$ \MSp(4) := \left\{ (X,\sigma(X)) \in \text{Mat}_{4 \times 4} \times \Ga : {}^* X X = \sigma(X) \right\} \hookrightarrow \text{Mat}_{4 \times 4}, $$
    or equivalently, the Zariski closure of $\GSp(4)$ in $\text{Mat}_{4 \times 4}$. It is a reductive algebraic monoid with unit group $\GSp(4)$ and inherits the $\mathfrak{o}_F$-structure from that of $\angles{\cdot|\cdot}$. Likewise one can define the reductive monoid $\MSp(2n)$ for $n \geq 1$. Hereafter we identify $\mathbbm{1}_{\MSp(4, \mathfrak{o}_F)}$ with its restriction to $G(F)$.
  \item Satake \cite{Sat63} and Shimura \cite{Shi63} tried to extend the Tamagawa-Godement-Jacquet construction to $\GSp(4)$ by taking the function $\mathbbm{1}_{\MSp(4, \mathfrak{o}_F)} \in \mathcal{H}_\text{ac}(G(F)\sslash K)$ instead of $f_{\text{spin}, s}$; its Satake transform is called the \textit{local Hecke series} of $G$ in \cite[Appendix 1]{Sat63}. Unlike the Tamagawa-Godement-Jacquet case, it turns out that $\mathbbm{1}_{\MSp(4, \mathfrak{o}_F)}$ is not basic enough: for a $K$-unramified irreducible representation $(\pi, V_\pi)$ of Satake parameter $c$, the case $\nu=2$ of \cite[Appendix 1, \S 3]{Sat63} says
  \begin{align*}
    \Tr(\mathbbm{1}_{\MSp(4, \mathfrak{o}_F)} \big| V_\pi) & = \frac{P(c)}{\det\left( 1 - q_F^{-\frac{3}{2}}\text{spin}(c) \right)} = P(c) \cdot L\left(-\frac{3}{2}, \pi, \text{spin}\right) \\
    & = P(c) \cdot \Tr\left( f_{\text{spin}, -\frac{3}{2}} \big| V_\pi \right)
  \end{align*}
  for some explicit $P \in \mathcal{H}(T(F)\sslash K_T)^W$ depending on $q_F$, upon twisting $\pi$ by some $|\sigma|_F^s$ for $\Re(s) \gg 0$ to make things converge. Moreover, $P \not\equiv 1$ as a function on $\hat{T}/W$. Equivalently,
  $$ \mathbbm{1}_{\MSp(4, \mathfrak{o}_F)} = \mathcal{S}^{-1}(P) * f_{\text{spin}, -\frac{3}{2}}. $$

  \item There is a conceptual explanation for the failure of $\mathbbm{1}_{\MSp(4, \mathfrak{o}_F)}$ to yield the spinor $L$-factor. In what follows, we cheat somehow by considering the case of generic $q_F$, i.e. $q_F = q$ will be viewed as an indeterminate. Let $s \in \Q$. If $f_{\text{spin},s}$ equals the restriction of $\mathbbm{1}_{\MSp(4, \mathfrak{o}_F)}$ to $G(F)$, then for each $\mu \in X_*(T)_-$ the coefficient of $\mathbbm{1}_{K\mu(\varpi)K}$ in $f_{\rho,s}$ would be either $0$ or $1$. Suppose this is the case for generic $q = q_F$, the same property will then hold for $q=1$. Now the coefficient of $\mathbbm{1}_{K\mu(\varpi)K}$ at $q=1$ equals $c_\mu(1)$, which is the cardinality of 
  $$ \left\{ (a_0, a_1, a_2, a_3) \in \Z_{\geq 0}^4 : a_0 (-\check{\alpha} - 2\check{\beta} + \check{\gamma}) + a_1 (-\check{\alpha} - \check{\beta} + \check{\gamma}) + a_2 (-\check{\beta}+\check{\gamma}) + a_3 \check{\gamma} = \mu \right\} . $$

  Since $X_*(T)$ is of rank $3$, the cardinality of this set has asymptotically polynomial growth in $\mu$ with degree $1$; see also the proof of Proposition \ref{prop:C-spin}. In particular, the $0$/$1$ dichotomy cannot hold.

  \item One can show the irreducibility of $\Sym^k(\text{spin})$ for all $k$ by Proposition \ref{prop:Sym-irred}, by imitating the proof Lemma \ref{prop:c_mu-Std}. This does not determine $c_\mu(q)$ for $(\GSp(4), \text{spin})$, since the weights of $\Sym^k(\text{spin})$ do not have multiplicity one in general.

  \item Some calculations for the Satake transforms of $\mathbbm{1}_{\MSp(2n, \mathfrak{o}_F)}$, as well as for some other classical similitude groups, are made in \cite{HS83} for $n \geq 2$: the non-trivial numerator $P$ is always present, and quickly becomes unmanageable.
\end{enumerate}

\bibliographystyle{abbrv}
\bibliography{basic}

\begin{flushleft}
  Wen-Wei Li \\
  Institute of Mathematics, \\
  Academy of Mathematics and Systems Science, Chinese Academy of Sciences, \\
  55, Zhongguancun East Road, \\
  100190 Beijing, People's Republic of China. \\
  E-mail address: \href{mailto:wwli@math.ac.cn}{\texttt{wwli@math.ac.cn}}
\end{flushleft}

\end{document}